\documentclass[10pt,a4paper]{article}
\usepackage[top=0.75in, bottom=0.75in, left=1in, right=1in]{geometry}
\usepackage{amsmath,amssymb}
\usepackage{algpseudocode}
\usepackage[section]{algorithm}
\usepackage{algorithmicx}
\usepackage{caption}
\usepackage{subcaption}
\usepackage{amsthm}
\numberwithin{equation}{section}
\usepackage{amsfonts}
\usepackage{graphicx}
\usepackage{hyperref}
\theoremstyle{plain}
\newtheorem{Thm}{Theorem}[section]
\newtheorem{Rem}{Remark}[section]

\newtheorem{Lem}{Lemma}[section]

\theoremstyle{definition}
\newtheorem{Defn}{Definition}[section]

\author{Thomas Y. Hou\thanks{hou@cms.caltech.edu, Computing and Mathematical Sciences, CalTech, Pasadena, CA 91125.}\and Pengfei Liu\thanks{Correspondence. plliu@caltech.edu, Computing and Mathematical Sciences, CalTech, Pasadena, CA 91125.}}
\date{}
\title{A Heterogeneous Stochastic FEM Framework for Elliptic PDEs}
\begin{document}
\maketitle 
\begin{abstract}
	We introduce a new concept of sparsity for the stochastic elliptic operator $-{\rm div}\left(a(x,\omega)\nabla(\cdot)\right)$, which reflects the compactness of its inverse operator in the stochastic direction and allows for spatially heterogeneous stochastic structure. This new concept of sparsity motivates a heterogeneous stochastic finite element method ({\bf HSFEM}) framework for linear elliptic equations, which discretizes the equations using the heterogeneous coupling of spatial basis with local stochastic basis to exploit the local stochastic structure of the solution space. We also provide a sampling method to construct the local stochastic basis for this framework using the randomized range finding techniques. The resulting HSFEM involves two stages and suits the multi-query setting: in the offline stage, the local stochastic structure of the solution space is identified; in the online stage, the equation can be efficiently solved for multiple forcing functions. An online error estimation and correction procedure through Monte Carlo sampling is given. Numerical results for several problems with high dimensional stochastic input are presented to demonstrate the efficiency of the HSFEM in the online stage.
\end{abstract}
\section{Introduction and Main Results}
Analysis of complex systems requires not only a fine understanding of the underlying physics, but also recognition of the intrinsic uncertainties and their influence on the quantities of interest. Uncertainty Quantification ({\bf UQ}) is an emerging discipline that aims at addressing the latter issue and has attracted growing interest recently. In this paper we consider UQ of the following linear elliptic equation with stochastic coefficient, which can be used to model diffusion processes in random media:
\begin{equation}
	\begin{cases}
		-{\rm div}\left(a(x,\omega)\nabla u(x,\omega)\right)=f(x),\quad x\in D, \quad \omega \in \Omega,\\
		u(x,\omega)|_{\partial D}=0.
	\end{cases}
	\label{targ}
\end{equation} 
Here $D$ is a bounded convex polygon domain in $R^d$, and $(\Omega, \mathcal{F}, P)$ is a probability space with $\Omega\subset R^m$, i.e., the dimension of the stochastic input $\omega$ is $m$. We assume that $f(x) \in L^2(D)$, and $a(x,\omega)$ is bounded and uniformly elliptic, i.e., there exist $\lambda_{\min}$ and $\lambda_{\max}$ such that
\begin{equation}
	\label{ellipticity}
P( \omega\in\Omega:\  a(x,\omega)\in [\lambda_{\rm min},\ \lambda_{\rm max}],\quad \forall x\in D)=1.
\end{equation}
The solution to equation~\eqref{targ} is $u(x,w)\in L^2(H^1_0(D),\Omega)$, such that for any $\phi(x,\omega)\in L^2(H^1_0(D),\Omega)$,
\begin{equation}\int_\Omega\int_D \nabla u(x,\omega)^Ta(x,\omega)\nabla \phi(x,\omega)\mathrm{d}x\mathrm{d}P =\int_\Omega\int_D f(x)\phi(x,\omega)\mathrm{d}x\mathrm{d}P,
\label{solu}
\end{equation} 
where the function space $L^2(H^1_0(D),\Omega)$ is a Hilbert space with inner product defined as
\begin{equation}
	\langle u(x,\omega),v(x,\omega)\rangle =\int_\Omega\int_D\nabla u(x,\omega)^T\nabla v(x,\omega)\mathrm{d}x\mathrm{d}P.
\end{equation}
The existence of solution to equation~\eqref{targ} can be obtained using the Lax-Milgram Theorem,
\begin{equation}
	\label{continuity}
	\|u(x,\omega)\|_{L^2(H^1_0(D),\Omega)}\leq C\|f(x)\|_{H^{-1}(D)},
\end{equation}
and more theoretical aspects of this equation can be found in \cite{babuska2004galerkin}.

Several types of numerical methods have been proposed for UQ of Stochastic Partial Differential Equations ({\bf SPDE}). Perturbation Methods ({\bf PM}) \cite{ghanem1991stochastic,babuvska2002perturb,kleiber1992stochastic}
start with expanding the stochastic solution via Taylor expansion and result in a system of deterministic equations by truncating after certain order. Typically at most second-order expansions are used as the system of equations becomes very cumbersome if higher-order terms are included. One limitation of perturbation methods is that the magnitude of the input and output uncertainty must be small compared with their respective means. Monte Carlo Methods ({\bf MC}) \cite{caflisch1998monte, niederreiter1992random} remain popular for SPDE problems because of its non-intrusive nature especially when the dimension of the stochastic input is high. According to the law of large numbers, the convergence rate of MC is only $O(M^{-1/2})$, where $M$ is the number of realizations. This low convergence rate limits its application. MC Methods are also sensitive to the random number generator and typically produce output with uncertain accuracy. Stochastic Collocation Methods ({\bf SC}) \cite{babuvska2010stochastic, nobile2008sparse, smolyak1963quadrature, mathelin2005stochastic,xiu2005high} collocate the problem in zeros of tensor product orthogonal polynomials and the solution is recovered using interpolation. SC methods attain high accuracy when the solution is smooth with respect to the random variables. However, a very large number of collocation points are required to obtain accuracy when the stochastic input has high dimension because of the tensor product. This is called the {\em curse of dimensionality}. Polynomial Chaos Methods ({\bf PC}) \cite{wiener1938homogeneous, cameron1947orthogonal,hou2006wiener, xiu2002wiener,frauenfelder2005finite}
project the solution $u(x,\omega)$ to an orthogonal polynomial basis  $H_\alpha(\omega)$ with respect to the underlying probability distribution, and approximate the solution by $u(x,\omega)\approx\sum_\alpha u_\alpha(x) H_\alpha(\omega)$. The coefficients $u_\alpha(x)$ can be obtained by solving a system of coupled elliptic equations. PC Methods also suffer from the {\em curse of dimensionality} because a large number of polynomial basis functions are required when the stochastic dimension is high.

SPDE problems with high stochastic dimension are very challenging because of the {\em curse of dimensionality}. There have been many attempts in the literature, e.g. \cite{cheng2013data, cheng2013dynamically1, cheng2013dynamically2,doostan2007stochastic,doostan2011non, math2012a, sargsyan2014dimensionality,bieri2009sparse,bieri2009sparse1, bieri2011sparse} to attack these challenging problems. Most of them take advantage of the fact that even though the stochastic input has high dimension, the solution actually lives in a relatively low dimensional space, i.e., enjoys some sense of sparsity. In \cite{doostan2011non, math2012a, sargsyan2014dimensionality} the compressive sensing technique is employed to identify a sparse representation of the solution in the stochastic direction. In \cite{cheng2013dynamically1, cheng2013dynamically2,doostan2007stochastic}, the Karhunen-Lo\`eve expansion of the solution is used to compactly represent the solution and reduce degrees of freedom. In \cite{ito1998reduced, rozza2008reduced}, the authors construct a reduced spatial basis using snapshots of the solutions to obtain computational savings for each realization of the stochastic equation. The present work also seeks to attack the {\em curse of dimensionality} by exploring the sparsity of the solution space. Our methodology differs from the previous sparsity-exploiting methods in the following two aspects: we use different stochastic basis functions in different regions of the domain to approximate the solution, allowing for spatially heterogeneous stochastic structure of the solution space;  we seek a sparse representation of the whole solution space for all $f(x)\in L^2(D)$, not a specific solution, thus our method suits the multi-query setting. 

In this work we first introduce a new concept of sparsity for the stochastic operator $-{\rm div}\left(a(x,\omega)\nabla(\cdot)\right)$. We consider a finite dimensional approximation to the solution space of \eqref{targ} taking the form of 
		\begin{equation}
			V_h=\{\sum_{i=1}^n\sum_{j=0}^{k_i}c_i^j\phi_i(x)\xi_i^j(\omega): c_i^j\in R\},
			\label{ftrialspace}
		\end{equation}
		where $\{\phi_i(x)\}$ is a standard piecewise linear basis that can resolve the spatial variation of the solution, and $\xi_i^j(\omega), j=0,\dots k_i$ are the {\bf local stochastic basis} functions associated with $\phi_i(x)$. Note that in~\eqref{ftrialspace}, different local stochastic basis functions are used in different regions of the domain to approximate the solution, and this allows for spatially heterogeneous stochastic structure of the solution space to \eqref{targ}. To obtain certain accuracy using \eqref{ftrialspace}, the required $k=\sum_{i=1}^n k_i/n$, which is the relative size of $V_h$ to the approximate solution space of the corresponding deterministic equation, measures the compactness of the inverse of the stochastic operator in the stochastic direction. If $k$ is {\em small}, then we say the stochastic operator enjoys the Operator-Sparsity. We call it {\bf weak} or {\bf strong} Operator-Sparsity depending on whether the approximation is taken to be in $L^2(D\times\Omega)$ or $L^2(H^1_0(D),\Omega)$. An interpretation of this Operator-Sparsity in terms of the decay rate of the singular values in the Karhunen-Lo\`eve expansion of the Green's function is given. We also prove that to obtain certain approximation accuracy using $V_h$ in $L^2(D\times\Omega)$, the required $k$ has an upper bound that only depends on the ellipticity of the operator, $\lambda_{\min}$ and $\lambda_{\max}$, and in particular, does not depend on the dimension of the stochastic input. 

		This new concept of sparsity motivates a heterogeneous stochastic finite element method framework ({\bf HSFEM}) for linear elliptic equations. We first construct a local stochastic basis $\{\xi_i^j(\omega)\}$ and trial space $V_h$ \eqref{ftrialspace}. Then based on \eqref{solu}, we define the numerical solution as $u_h(x,\omega)\in V_h$, such that
		\[\int_\Omega\int_D \nabla u_h(x,\omega)^Ta(x,\omega)\nabla v(x,\omega)\mathrm{d}x\mathrm{d}P=\int_\Omega\int_D f(x)v(x,\omega)\mathrm{d}x\mathrm{d}P,\ \ \forall v(x,\omega)\in V_h.\]
The numerical solution defined as above satisfies the following quasi-optimality property
		\[\|u(x,\omega)-u_h(x,\omega)\|_{L^2(H^1_0(D),\Omega)}\leq \lambda_{\rm max}/\lambda_{\rm min}\inf_{v\in V_h}\|u(x,\omega)-v(x,\omega)\|_{L^2(H^1_0(D),\Omega)}.\]
		The key difference of this HSFEM framework from classical methods like PC or SC methods is that different local stochastic basis functions are used in different regions of the domain to discretize the equation, which allows for spatially heterogeneous stochastic structure of the solution space. If the operator enjoys the (strong) Operator-Sparsity, this HSFEM framework can attain high accuracy using only a small trial space $V_h$. This HSFEM framework can be viewed as a generalization of the PC methods using {\bf problem-dependent} and {\bf local} stochastic basis, and provides a novel direction to attack the {\em curse of dimensionality} by exploiting the local stochastic structure of the solution space.

		We also provide a sampling method to construct the local stochastic basis for the HSFEM framework using the randomized range finding techniques. Given the HSFEM framework, a suitable local stochastic basis $\{\xi_i^j(\omega)\}$ still need to be constructed. Since the spatial basis $\phi_i(x)$ has local support around the node point $x_i$, the local stochastic basis functions $\xi_i^j(\omega), j=0,\dots,k_i$ in the trial space \eqref{ftrialspace} are only used to approximate the solution in the stochastic direction near $x_i$. So we consider the following linear operator which maps the forcing function to the stochastic part of the solution at $x_i$,
		\[T_i:\quad L^2(D)\to L^2(\Omega),\quad f(x)\to u(x_i,\omega)-E[u(x_i,\omega)].\]
$T_i$ can be shown to be a compact linear operator thus can be approximated by a matrix. We apply the randomized range finding techniques~\cite{halko2011finding} to $T_i$ to construct the local stochastic basis functions $\xi_i^j(\omega), j=1,\dots,k_i$. The resulting local stochastic basis functions are close to optimal in approximating the range of $T_i$ in the $L^2(\Omega)$ sense if the singular values of $T_i$ decay fast.

The HSFEM framework combined with the sampling method to construct the local stochastic basis constitutes a heterogeneous stochastic finite element method. This method involves two stages: 
		\begin{itemize}
			\item In the offline stage we identify the sparse structure of the solution space. We construct the local stochastic basis by solving the SPDE for a number of randomly chosen forcing functions followed by some orthogonalization process, and then build the corresponding stiffness matrix. 
			\item In the online stage, the equation can be efficiently solved for multiple forcing functions using the stiffness matrix constructed in the offline stage, which is small and sparse.
		\end{itemize}
		The offline computation of the HSFEM is expensive since it involves solving the SPDE for a number of times using traditional methods, and this limits its application to the multi-query setting, which means the equation needs to be solved multiple times using different forcing functions. Our method suits the multi-query setting because the linear system we solve online is sparse and small. Methods to reduce the offline computation cost and an online error estimation and correction procedure are given. 
		
We present numerical results for several problems with high dimensional stochastic input to demonstrate the efficiency of the HSFEM. In our numerical examples, we show that the solution space to stochastic elliptic equation has spatially heterogeneous stochastic structure, and this heterogeneity can be recognized by the HSFEM in the offline stage. Moreover, the number of the constructed local stochastic basis functions is small and the online numerical solutions have high accuracy, which implies that the stochastic operators that we consider enjoy the Oprator-Sparsity and the local stochastic basis constructed using the sampling method works well within the HSFEM framework. We also compare the accuracy of the HSFEM numerical solution with the truncated Karhunen-Lo\`eve expansion of the solution to demonstrate the advantage of approximating the stochastic part of the solution locally.

The rest part of this paper is organized as follows. In section~\ref{sparsitysection}, we introduce the Karhunen-Lo\`eve expansion and the new concept of sparsity. In section \ref{hsfemsection}, we develop the HSFEM framework. In section~\ref{samplingsection}, we discuss the sampling method to construct the local stochastic basis. In section \ref{implementationsection}, we address several issues related to the numerical implementation of the HSFEM. In section \ref{examples}, we present our numerical results. Section~\ref{concludingsection} is devoted to concluding remarks and future work.


\section{A New Concept of Sparsity for Stochastic Elliptic Operator}
\label{sparsitysection}
In this section, we first introduce the Karhunen-Lo\`eve~({\bf KL}) expansion \cite{ghanem1991stochastic, stark1986probability} and the sense of sparsity related to the KL expansion. Two observations are made on the KL expansion, which motivate an Operator-Sparsity allowing for spatially heterogeneous stochastic structure of the corresponding solution space. An interpretation of this new concept of sparsity in terms of the Green's function is given.   

\subsection{Karhunen-Lo\`eve Expansion}
\label{klsection}
If a stochastic process $u(x,\omega)\in L^2(D\times \Omega)$, then it can be expanded in a Fourier-type series as:\cite{ghanem1991stochastic}
\begin{equation}
	\label{kLexpansion}
	u(x,\omega)=\bar{u}(x)+\sum_{i=1}^\infty\sqrt{\lambda_i}\psi_i(x)\xi_i(\omega),\quad \lambda_1\geq \lambda_2\geq \dots\geq\lambda_i\geq \dots > 0,
\end{equation}
where $\bar{u}(x)=\int_\Omega u(x,\omega)\mathrm{d}P$, $\psi_i(x)$ are orthonormal in $L^2(D)$ and $\xi_i(\omega)$ are uncorrelated uni-variate random variables with mean $0$. \eqref{kLexpansion} is called the Karhunen-Lo\`eve ({\bf KL}) expansion of $u(x,\omega)$, and $\psi_i(x)$ can be computed as eigenfunctions of the covariance function of $u(x,\omega)$. The $k$-term truncation 
\begin{equation} 
	\label{klapprox}
\sum_{i=1}^k\sqrt{\lambda_i}\psi_i(x)\xi_i(\omega)
\end{equation}
is the best rank-$k$ approximation of $u(x,\omega)-\overline{u}(x)$ in $L^2(D\times \Omega)$.\cite{ghanem1991stochastic}

If the singular values $\sqrt{\lambda_i}$ decay very fast in the KL expansion, then the solution has data-sparsity in the sense that a few data can provide an accurate description of the solution. More discussion about the data-sparsity can be found in \cite{hackbusch1999sparse,hackbusch2000sparse}. This sense of sparsity has been exploited to reduce the computation cost in \cite{cheng2013data, cheng2013dynamically1, cheng2013dynamically2, doostan2007stochastic}, and we refer it as the KL-sense Sparsity.

\begin{Defn}[KL-sense Sparsity]
	The solution $u(x,\omega)$ to equation~\eqref{targ} is said to have KL-sense sparsity if the singular values $\sqrt{\lambda_i}$ in its KL expansion decay very fast (e.g. exponential decay), and a small number of modes are sufficient to represent the solution accurately.
\end{Defn}
\begin{Rem}
	\label{fast}
Here we leave the notions `decay very fast' and `accurately' vague, because they depend on the nature of the problem and the desired order of accuracy.
\end{Rem}

We make the following two observations on the KL expansion and the KL-sense sparsity:
\begin{itemize}
	\item The KL-sense sparsity reflects the property of some solution to equation~\eqref{targ} with a specific forcing, not the property of the solution space. The first several stochastic basis functions $\xi_i(\omega)$ in the KL expansion of one specific solution do not necessarily approximate another solution well.
	\item The truncated KL expansion \eqref{klapprox} seeks to approximate the stochastic behavior of the solution in different regions of the domain using the same set of random variables $\xi_i(\omega), i=1,\dots, k$, which is quite restrictive since the solution may have spatially heterogeneous stochastic structure.
\end{itemize}
These two observations motivate us to define the following Operator-Sparsity.

\subsection{A New Concept of Sparsity for the Stochastic Operator}
We consider the inverse of the stochastic elliptic operator $L(x,\omega)=-{\rm div}\left(a(x,\omega)\nabla(\cdot)\right)$, which is denoted as $L^{-1}(x,\omega)$. Then $L^{-1}(x,\omega)$ maps $f(x)\in L^2(D)$ to $u(x,\omega)$, which is the solution to equation~\eqref{targ}.
\begin{Rem}[Random Forcing Functions]
	In this paper, we only consider deterministic forcing $f(x)$. For the case that the right hand side of \eqref{targ} has randomness which we denote by $f(x,\omega)$, we can first decompose $f(x,\omega)$ (for example, using KL expansion) to
\[f(x,\omega)\approx f_0(x)+\sum_{i=1}^{N_f}f_i(x)\xi_i(\omega).\]
Then based on \eqref{targ}, the corresponding solution can be approximated using
\begin{equation*}
	u(x,\omega)\approx L^{-1}(x,\omega)f_0(x)+\sum_{i=1}^{N_f}\xi_i(\omega)L^{-1}(x,\omega)f_i(x),
	\label{randomforce}
\end{equation*}
which means once we have obtained an efficient solver for equation~\eqref{targ} with deterministic forcing, we can also solve the equation with random forcing efficiently.
\end{Rem}

We seek to construct a finite dimensional linear operator $L_h^{-1}(x,\omega)$ to approximate $L^{-1}(x,\omega)$. It is known that the Lagrange basis (piecewise polynomials) can approximate the solution to deterministic elliptic equation very well \cite{steinbach2008} if the coefficient $a(x)$ is in $C^1(D)$. For the SPDE, it is natural to use the product ({\bf not} tensor product) of these Lagrange basis functions with some stochastic basis to approximate the solution to equation~\eqref{targ}. To be specific: assume $D$ is equipped with a triangular mesh with node points $x_i, i=1,\dots, n$, and $\phi_i(x)$ are the piecewise linear functions on $D$ satisfying \[\phi_i(x_j)=\delta_{ij}.\]
Then we want to construct a stochastic basis $\{\xi_i^j(\omega)\}$, and use
\begin{equation} 
	\label{trialspace}
		V_h=\{\sum_{i=1}^n\sum_{j=0}^{k_i}c_i^j\phi_i(x)\xi_i^j(\omega): c_i^j\in R\}
\end{equation}
to approximate the solution space of equation~\eqref{targ}. Note that $V_h$ consists of the heterogeneous coupling of spatial basis with stochastic basis, namely, we use different stochastic basis functions to couple with different spatial basis functions to approximate the solution space to \eqref{targ}, and this allows for spatially heterogeneous stochastic structure of the solution space. We want
\begin{equation}
	\label{sparsity}
	\sup_{f(x)\in L^2(D)}\frac{ \inf_{c_i^j\in R}\|u(x,\omega)-\sum_{i=1}^n\sum_{j=0}^{k_i}c_i^j\phi_i(x)\xi_i^j(\omega)\|_H}{\|f(x)\|_{L^2(D)}}\leq \epsilon,
\end{equation}
where $\|\cdot\|_H$ can be taken to be the $L^2(H^1_0(D),\Omega)$ norm or $L^2(D\times\Omega)$ norm.

We call $\{\phi_i(x)\xi_i^j(\omega)\}$ the {\bf Coupling Basis}. Assume that to obtain {\bf certain order of accuracy}, the total number of the coupling basis functions in the above approximation~\eqref{sparsity} is 
\begin{equation*}
	S=\sum_{i=1}^n (k_i+1).
\end{equation*} 

The ratio of $S$ over $n$ is the relative size of the approximate solution space of \eqref{targ}, $V_h$, to the corresponding deterministic equation, thus measures the compactness of the inverse operator $L^{-1}(x,\omega)$ in the stochastic direction. This motivates us to define the following sense of sparsity.
\begin{Defn}[Operator-Sparsity]\label{defos}
	For the stochastic elliptic operator 
	\begin{equation*}
		L(x,\omega)=-{\rm div}\left(a(x,\omega)\nabla(\cdot)\right),
	\end{equation*}
	if the inverse of $L(x,\omega)$ can be well-approximated using finite dimensional operator in the sense of \eqref{sparsity}, and the average number of stochastic basis functions for each node point is small, i.e.,
	\begin{equation*}
		k=\frac{1}{n}\sum_{i=1}^nk_i \quad \text{is small},
	\end{equation*}
then we say that $L(x,\omega)$ has Operator-Sparsity.

If the approximation is taken to be in the $L^2(H^1_0(D),\Omega)$ sense, we call it the {\bf strong} Operator-Sparsity. Or if the approximation is taken to be in the $L^2(D\times\Omega)$ sense, we call it the {\bf weak} Operator-Sparsity.
\end{Defn}
\begin{Rem}
	We leave the notions `well-approximated' and `small' vague since they depend on the desired order of accuracy and the nature of the problem.
\end{Rem}
\subsection{Interpretation of the Operator-Sparsity}
\label{interpretation}
Since $\phi_i(x)$ has compact support near $x_i$, the $\xi_i^j(\omega)$, $0=1,\dots, k_i$, in \eqref{sparsity} are only used to approximate the solution in the stochastic direction near $x_i$. So we call $\{\xi_i^j(\omega)\}$ the {\bf Local Stochastic Basis}.

To give an interpretation of this new concept of sparsity, we consider the local stochastic behavior of the solution near $x_i$. Let $G(x,y,\omega)$ be the Green's function corresponding to the elliptic operator $-{\rm div} \left(a(x, \omega) \nabla (\cdot) \right)$, then the solution to equation~\eqref{targ} at $x_i$ is
\begin{equation*}
	u(x_i,\omega)=\int_D G(x_i,y,\omega)f(y)\mathrm{d}y.
\end{equation*}
For $d \leq 3$, based on the estimates of Green's function in \cite{gruter1982green, taylor2013green}, we have that for fixed $x_i$ and $\omega\in\Omega$,
\begin{equation}
	\label{greenestimate}
	\|G(x_i,y,\omega)\|_{L^2(D)}\leq C_i,
\end{equation}
where $C_i$ depends on $x_i$, $\lambda_{\min}$, $\lambda_{\max}$ and the domain $D$, but does not depend on the dimension of $\omega$. Integrating \eqref{greenestimate} in $\Omega$, we get
\[\|G(x_i,y,\omega)\|_{L^2(D\times\Omega)}\leq C_i<+\infty.\]

So the Green's function $G(x_i,y,\omega)$ has KL expansion
\begin{equation}
	G(x_i,y,\omega)=\bar G(x_i,y)+\sum_{j=1}^\infty\sigma_i^j\psi_i^j(y)\eta_i^j(\omega),\quad \sum_{j=1}^{\infty}(\sigma_i^j)^2\leq C_i^2,
	\label{KLgreen}
\end{equation}
where $\psi_i^j(y)$ and $\eta_i^j(\omega)$, $j=1,\dots, \infty$ are orthonormal in $L^2(D)$ and $L^2(\omega)$ respectively. Then 
\begin{equation}
	u(x_i,\omega)-\bar u(x_i)=\sum_{j=1}^\infty \sigma_i^j \eta_i^j(\omega)\int_D f(y)\psi_i^j(y)\mathrm{d}y.
	\label{compactevi}
\end{equation}

Denote $k_i$ as the number of required terms to make the singular values $\sigma_i^j$ decay to $\epsilon$, i.e.,
\begin{equation}
	\label{minki}
	k_i=\min\{j:\sigma_i^j<\epsilon\}.
\end{equation}

Since $\sum_{j=1}^\infty(\sigma_i^j)^2\leq C_i^2$, we have
\begin{equation}
k_i\leq [C_i^2/\epsilon^2],
	\label{worst}
\end{equation}
where the right hand side does not depend on the dimension of the stochastic input $\omega$. 

When the first $k_i$ stochastic basis functions are used to approximate the stochastic solution, we have 
\begin{equation}
		\label{greenapp}
		\|u(x_i,\omega)-\bar u(x_i)-\sum_{j=1}^{k_i}\sigma_i^j\eta_i^j(\omega)\int_Df(y)\psi_i^j(y)\mathrm{d}y\|_{L^2(\Omega)}\leq \sigma_i^{k_i+1}\|f(x)\|_{L^2(\Omega)}\leq \epsilon\|f(x)\|_{L^2(D)}.
\end{equation}

\begin{Rem}
	The estimate \eqref{worst} corresponds to the worst case that the first $[C_i^2/\epsilon^2]$ singular values in the KL expansion of $G(x_i,y,\omega)$ are equal to $\epsilon$ and others are $0$. It is a very pessimistic estimate, and according to our numerical results in section~\ref{compactsection}, the singular values $\sigma_i^j$ actually decay exponentially fast (for that particular example), which means there exist $c_i, C_i>0$, such that
	\begin{equation*}
		\sigma_i^j\leq C_ie^{-c_ij}.
	\end{equation*}
	Then to obtain $\epsilon$-accuracy in approximating the stochastic part of the solution at $x_i$, only
	\begin{equation}
		k_i=[\frac{\log C_i-\log\epsilon}{c_i}],
		\label{mildgrowth}
	\end{equation}
	stochastic basis functions are needed. The decay rate of the singular values in the KL expansion of the Green's function $G(x_i,y,\omega)$ will be investigated in another work.
\end{Rem}

So if the singular values in the KL expansion of $G(x_i,y,\omega)$, $\sigma_i^j$, decay fast, a small number of stochastic basis functions $\eta_i^j(\omega), j=1,\dots, k_i$ can approximate the local stochastic behavior of the solution near $x_i$ well. This is very close to (but not exactly the same as) our definition of Operator-Sparsity, so we can interpret the Operator-Sparsity as the {\bf low-rankness} of the Green's function $G(x,y,\omega)$ for fixed $x$.
	
To make this interpretation rigorous, we denote $Ju(x,\omega)$ as the piecewise linear interpolation of the solution $u(x,\omega)$ using the spatial basis $\phi_i(x)$, which means $Ju(x_i,\omega)=u(x_i,\omega)$. We assume that the spatial basis $\phi_i(x)$, $i=1,\dots, n$ can resolve the variation of the solution in the spatial direction, then 
\begin{equation}
	\label{negligible}
	\|u(x,\omega)-Ju(x,\omega)\|_{L^2(D\times\Omega)}
\end{equation}
is small and we ignore this interpolation error. We assume that the domain is divided into elements $\tau_i$, $i=1, \dots, n_t$, which are triangles for $d=2$ and tetrahedrons for $d=3$, and let $x_i^j$ be the nodes of $\tau_i$. 
\begin{Lem}
	\label{nouse}
	Choose the first $k_i$ \eqref{minki} stochastic basis functions in the KL expansion of $G(x_i,y,\omega)$, $\eta_i^j(\omega)$, $j=1,\dots, k_i$ as the local stochastic basis in \eqref{trialspace} and let $\xi_i^0(\omega)=1$. Then 
\begin{equation}
	\inf_{v(x,\omega)\in V_h}\|u(x,\omega)-v(x,\omega)\|_{L^2(D\times\Omega)}\leq C\epsilon\|f(x)\|_{L^2(D)}.\label{Lestimate}
\end{equation}
\end{Lem}
\begin{proof}
Let $c_i^0=\bar u(x_i)$, $c_i^j=\sigma_i^j\int_Df(y)\psi_i^j(y)\mathrm{d}y$ for $j>0$, and denote $Pu(x,\omega)\in V_h$ as
\begin{equation*}
	Pu(x,\omega)=-\sum_{i=1}^n\sum_{j=0}^{k_i}c_i^j\phi_i(x)\eta_i^j(\omega).
\end{equation*}
Then 
\begin{equation*}
	Ju(x_i,\omega)-Pu(x_i,\omega)=u(x_i,\omega)-\bar u(x_i)-\sum_{j=1}^{k_i}\sigma_i^j\eta_i^j(\omega)\int_Df(y)\psi_i^j(y)\mathrm{d}y.
\end{equation*}
According to \eqref{greenapp},
\begin{equation}
	\label{simplifiednote}
	\|Ju(x_i,\omega)-Pu(x_i,\omega)\|_{L^2(\Omega)}\leq \epsilon\|f(x)\|_{L^2(D)}.
\end{equation}
Consider the $L^2(D\times\Omega)$ norm of the error,
\begin{equation}
	\label{L2expand}
	\|Ju(x,\omega)-Pu(x,\omega)\|^2_{L^2(D\times\Omega)}=\int_\Omega\|Ju(x,\omega)-Pu(x,\omega)\|^2_{L^2(D)}\mathrm{d}P.
\end{equation}
Since for fixed $\omega$, $Ju(x,\omega)-Pu(x,\omega)$ is piecewise linear, we have
\begin{align}
	\|Ju(x,\omega)-Pu(x,\omega)\|^2_{L^2(D)}&=\sum_{i=1}^{n_t}\|Ju(x,\omega)-Pu(x,\omega)\|_{L^2(\tau_i)}^2\\
	&\leq \sum_{i=1}^{n_t}|\tau_i|(\sum_{j}(Ju(x_i^j,\omega)-Pu(x_i^j,\omega))^2.\label{L2crude}
\end{align}
Putting the estimate \eqref{L2expand} in \eqref{L2crude}, we get
\begin{equation*}
	\|Ju(x,\omega)-Pu(x,\omega)\|_{L^2(D\times\Omega)}^2\leq \sum_{i=1}^{n_t}|\tau_i|\sum_j \|Ju(x_i^j,\omega)-Pu(x^j_i,\omega)\|_{L^2(\Omega)}^2.
\end{equation*}
According to \eqref{simplifiednote}, we get 
\begin{equation*}
	\|Ju(x,\omega)-Pu(x,\omega)\|_{L^2(D\times\Omega)}\leq C\epsilon\|f(x)\|_{L^2(D)}.
\end{equation*}
With \eqref{negligible}, we complete the proof.
\end{proof}
\begin{Rem}
	To obtain a similar estimate for the approximation error of $V_h$ in the $L^2(H^1_0(D),\Omega)$ sense, we need the regularity of the error in the spatial direction. We leave this issue to our future work.
\end{Rem}

Based on \eqref{worst} and estimate \eqref{Lestimate}, we have that to obtain certain approximation accuracy of the solution space to \eqref{targ} using finite dimensional space \eqref{trialspace} in $L^2(D\times\Omega)$, the required number of local stochastic basis functions $k_i$ has an upper bound that only depends on the ellipticity of the stochastic operator, $\lambda_{\min}$ and $\lambda_{\max}$, and in particular, does not depend on the dimension of the stochastic input. 

Before we go to the next section, we make the following observations on this Operator-Sparsity.
\begin{itemize}
	\item The Operator-Sparsity reflects the compactness of $L^{-1}(x,\omega)$ in the stochastic direction, but it is {\bf stronger} than compactness because we have restricted the finite dimensional space approximating the range of $L^{-1}(x,\omega)$ to take the form of \eqref{trialspace}.   
	\item The Operator-Sparsity permits using different stochastic basis functions in different regions of the domain to approximate the solution thus allows for spatially heterogeneous stochastic structure of the solution space. In this sense it is {\bf weaker} than the KL-sense Sparsity.
	\item The Operator-Sparsity reflects the property of the stochastic elliptic operator $-{\rm div}\left(a(x,\omega)\nabla(\cdot)\right)$ independent of the forcing function, thus it can be naturally exploited to solve equation~\eqref{targ} for multiple forcing functions and significantly reduce the computation cost.
\end{itemize}

Our definition of Operator-Sparsity \eqref{sparsity} together with the weak formulation of the stochastic elliptic equation~\eqref{solu} leads to a heterogeneous stochastic finite element method framework ({\bf HSFEM}) to solve linear stochastic elliptic equations, which is discussed in detail in the next section.

\section{A Heterogeneous Stochastic FEM Framework}
\label{hsfemsection}
Define a functional $J$ on $L^2(H^1_0(D),\Omega)$ as
\begin{equation*}
	J(u(x,\omega))=\frac{1}{2}\int_\Omega\int_D\nabla u(x,\omega)^Ta(x,\omega)\nabla u(x,\omega)\mathrm{d}x\mathrm{d}P-\int_\Omega\int_Du(x,\omega)f(x)\mathrm{d}x\mathrm{d}P.
\end{equation*}
Then we have
\begin{Thm}[Variational Formulation]The weak formulation of the stochastic elliptic equation~\eqref{solu} is equivalent to finding 
	\begin{equation*}
		u(x,\omega)\in L^2(H^1_0(D),\Omega),
	\end{equation*}
such that 
\begin{equation}
	\label{variation}
	J(u(x,\omega))\leq J(v(x,\omega)), \ \ \forall v(x,\omega)\in L^2(H^1_0(D),\Omega).
\end{equation}
\end{Thm}

Based on \eqref{variation}, given a finite dimensional subspace of $L^2(H^1_0(D),\Omega)$, $V_h$, the corresponding finite element formulation of equation~\eqref{targ} is finding $u_h(x,\omega)\in V_h$ such that
\begin{equation}
	\label{fem}
	J(u_h(x,\omega))\leq J(v(x,\omega)), \ \ \forall v(x,\omega)\in V_h.
\end{equation}
\begin{Thm}[Finite Element Formulation]
	The finite element formulation~\eqref{fem} is equivalent to finding 
	\begin{equation*}
		u_h(x,\omega)\in V_h,
	\end{equation*}
	such that
\begin{equation}
	\label{ortho}
\int_\Omega\int_D \nabla u_h(x,\omega)^T a(x,\omega) \nabla v(x,\omega) \mathrm{d}x\mathrm{d}P=\int_\Omega\int_D f(x)v(x,\omega)\mathrm{d}x\mathrm{d}P.\quad \forall v\in V_h.
\end{equation}
\end{Thm}
The numerical solution defined in this way has the following property.
\begin{Thm}[Quasi-Optimality]
\begin{equation} 
\label{optimality}
\|u(x,\omega)-u_h(x,\omega)\|_{L^2(H^1_0(D),\Omega)}\leq \lambda_{\max}/\lambda_{\min}\inf_{v\in V_h}\|u(x,\omega)-v(x,\omega)\|_{L^2(H^1_0(D),\Omega)}.
\end{equation}
\end{Thm}
If we choose the finite dimensional subspace of $L^2(H^1_0(D),\Omega)$ to be of the form~\eqref{trialspace}, i.e., the trial space $V_h$ is chosen to be the heterogeneous coupling of spatial basis with local stochastic basis, then we get the following HSFEM framework for solving linear stochastic elliptic equations. 
\begin{center}
\fbox{\parbox{0.98\textwidth}{
\begin{center} {\bf The Heterogeneous Finite Element Method Framework}
\end{center}
\begin{itemize}
\item Construct the local stochastic basis $\xi_i^j(\omega)$, $i=1,\dots, n$, $j=0,\dots, k_i$, which can capture the stochastic behavior of the solution near $x_i$. Let $\xi_i^0(\omega)=1$.
\item Find $c_i^j, i=1,\dots, n, j=0,\dots, k_i$, such that for all $i'=1,\dots, n$, $j'=0,\dots, k_{i'}$,
\begin{equation*}
	\int_\Omega\int_D\sum_{i=1}^n\sum_{j=0}^{k_i}c_i^j\xi_i^j\nabla\phi_i(x)^Ta(x,\omega)\xi_{i'}^{j'}\nabla\phi_{i'}(x)\mathrm{d}x\mathrm{d}P=\int_\Omega\int_D f(x)\xi_{i'}^{j'}\phi_{i'}(x)\mathrm{d}x\mathrm{d}P.
\end{equation*}
The corresponding linear system is uniquely solvable if for each $i$, $\xi_i^j(\omega)$, $j=0,\dots, k_i$ are linearly independent, i.e.,
\[\sum_{j=0}^{k_i} c_j\xi_i^j(\omega)=0\quad \Leftrightarrow\quad  c_j=0.\]

\item Recover the numerical solution as
\begin{equation*}
	u_h(x,\omega)=\sum_{i=1}^n\sum_{j=0}^{k_i}c_i^j\phi_i(x)\xi_i^j(\omega).
\end{equation*}
\item {\em A posteriori} error estimation and correction.
\end{itemize}}}
\end{center}

The novelty of the HSFEM framework is that the equation is discretized in the stochastic and spatial directions simultaneously using the heterogeneous coupling of spatial basis with local stochastic basis, and this allows for spatially heterogeneous stochastic structure of the solution space. Given this HSFEM framework, a suitable local stochastic basis still need to be chosen. Different choices of $\{\xi_i^j(\omega)\}$ lead to different numerical formulations of equation~\eqref{targ}. For example, if $\xi_i^j(\omega)$ are chosen to be the orthonormal polynomials $H_\alpha(\omega)$ in the PC methods for all $i$, then the HSFEM is equivalent to the PC \cite{babuska2004galerkin, xiu2002wiener} methods. This implies that the HSFEM framework is a generalization of the PC methods using {\bf problem-dependent} and {\bf local} stochastic basis to approximate the solution space. If the stochastic operator enjoys the (strong) Operator-Sparsity and a suitable local stochastic basis $\{\xi_i^j(\omega)\}$ is chosen, then based on \eqref{sparsity} and \eqref{optimality}, this HSFEM framework can attain high accuracy using only a small number of coupling basis functions $S=\sum_i^n k_i+n$. So the HSFEM framework provides a novel direction to attack the {\em curse of dimensionality} by exploiting local stochastic structure of the solution space.  

In the following section, we provide a sampling method to construct the local stochastic basis $\{\xi_i^j(\omega)\}$ using the randomized range finding techniques \cite{halko2011finding}. The constructed local stochastic basis functions are close to optimal in capturing the local stochastic behavior of the solution space in the $L^2(\Omega)$ sense if the singular values in the KL expansion of the Green's function decay fast.

\section{A Sampling Method to Construct the Local Stochastic Basis}
\label{samplingsection}
In this section we first introduce the compact linear operator $T_i$, which maps the forcing $f(x)$ to the stochastic part of the solution at node point $x_i$. We apply the randomized range finding techniques to $T_i$ to construct the local stochastic basis. Numerical implementation details are given.
\subsection{The Linear Compact Operator $T_i$}
\label{Sampling}
Since the local stochastic basis functions $\xi_i^j(\omega), j=1,\dots, k_i$ are used to approximate the behavior of the solution in the stochastic direction near $x_i$, we consider the following linear operator that maps the forcing function to the stochastic part of the solution at $x_i$: 
\begin{equation}
	\label{Ti}
	T_i: \quad f(x)\quad \mapsto \quad u(x_i, \omega)-E[u(x_i,\omega)].
\end{equation}
In dimension $d\leq 3$, $T_i$ is actually a compact linear operator mapping from $L^2(D)$ to $L^2(\Omega)$, whose singular value decomposition has already been given by \eqref{compactevi},
\begin{equation*}
	T_i f(y)=\sum_{j=1}^{\infty}\sigma_i^j\eta_i^j(\omega)\int_D \psi_i^j(y)f(y)\mathrm{d}y,
\end{equation*}
where $\sigma_i^j$, $\eta_i^j(\omega)$ and $\psi_i^j(x)$ come from the KL expansion of the Green's function~\eqref{KLgreen}. 
The subspace of $L^2(\Omega)$ spanned by the first $k$ left singular vectors of $T_i$ 
\begin{equation}
	\label{optimall2}
	{\rm span}\{\eta_i^1(\omega),\eta_i^2(\omega),\dots, \eta_i^k(\omega)\}
\end{equation}
is the best $k$ dimensional approximation of the range of $T_i$ in the $L^2(\Omega)$ sense. So $\eta_i^j(\omega)$, $j=1,\dots, k_i$, are good candidates for the local stochastic basis functions $\xi_i^j(\omega)$, $j=1,\dots, k_i$. 

Based on our analysis of the Green's function in section~\ref{interpretation}, we have $\sum_{j=1}^\infty(\sigma_i^j)^2<C_i^2$. So to obtain $\epsilon$ accuracy in approximating $T_i$ in the $L^2(\Omega)$ sense, at most $[C_i^2/\epsilon^2]$ (\ref{worst}) stochastic basis functions are required, which depends on the ellipticity of the operator, but is independent of the dimension of $\omega$. 

However, it is impractical to construct ${\rm span}\{\eta_i^1(\omega),\eta_i^2(\omega),\dots,\eta_i^k(\omega)\}$ by doing KL expansion to the Green's function. In this section we develop a sampling method using the randomized range finding algorithm to construct the local stochastic basis, which can capture the main action of $T_i$ and is an approximation of \eqref{optimall2}. The basic idea is using $T_i$ to act on some random matrix $\Omega$, then the main action of $T_i$ can be extracted through an orthogonalization process from the image $T_i\Omega$.
\subsection{Discretization of $T_i$}
We first discretize the domain of $T_i$, $L^2(D)$ using a set of orthonormal basis functions
\begin{equation}
	\label{basis}
	\hat D=\{\Phi_1(x), \Phi_2(x)\dots,\Phi_N(x) \}.
\end{equation} 

In our numerical examples in section~\ref{examples}, the domain $D$ is chosen to be $[0,1]^d$, and we discretize $D$ using uniform mesh of size $h$. $\hat D$ is chosen to be the first $N=(2l+1)^d$ Fourier modes with $l=\frac{1}{2h}$, i.e.,
\begin{equation}
	\hat{D}=\otimes_{i=1}^d\{1,\dots, 2\sin(2\pi lx_i), 2\cos(2\pi lx_i)\},
	\label{discretizeL2D}
\end{equation}
which contains all the Fourier modes that can be resolved by the given mesh.

\begin{Rem}
	$L^2(D)$ is an infinite dimensional space and cannot be approximated using finite basis functions. However, in this paper we seek to construct finite dimensional space $V_h$~\eqref{trialspace} to approximate the solution space to~\eqref{targ}. Assuming now the constructed space $V_h$ can make \eqref{sparsity} hold for all $f(x)\in \hat D$, i.e., for any $f(x)\in{\hat D}$, there exists $u_h(x,\omega)\in V_h$, such that 
\begin{equation*}
	\|u(x,\omega)-u_h(x,\omega)\|_H\leq\epsilon\|f(x)\|_{L^2(D)}.
\end{equation*}
where $\|\cdot\|_H$ can be either $L^2(D\times\Omega)$ or $L^2(H^1_0(D),\Omega)$ norm. Then for any $f(x)\in L^2(D)$, we have $f(x)=f^1(x)+f^2(x)$, where $f^1(x)\in \hat D$ and $f^2(x)\in \hat D^{\perp}$. $\hat D^{\perp}$ is the orthogonal complement of $\hat D$ in $L^2(D)$ and consists of high frequency Fourier modes. Then for $f^2(x)\in \hat D^{\perp}$, $\|f^2(x)\|_{H^{-1}(D)}\leq C/l\|f^2(x)\|_{L^2(D)}\leq Ch\|f^2(x)\|_{L^2(D)}$ according to \eqref{discretizeL2D} and the fact that we have chosen $l=1/2h$. 

The solution $u(x,\omega)$ can be divided into $u(x,\omega)=u^1(x,\omega)+u^2(x,\omega)$ correspondingly, and there exists $u^1_h(x,\omega)\in V_h$ such that $\|u^1(x,\omega)-u^1_h(x,\omega)\|_H\leq \epsilon \|f^1(x)\|_{L^2(D)}$. Then
\begin{equation*}
	\|u(x,\omega)-u^1_h(x,\omega)\|_{H}\leq \epsilon\|f^1(x)\|_{L^2(D)}+\|u^2(x,\omega)\|_H,
\end{equation*}
and according to \eqref{continuity}, $\|u^2(x,\omega)\|_H\leq C\|f^2(x)\|_{H^{-1}(D)}\leq Ch\|f^2(x)\|_{L^{2}(D)}$. Then
\begin{equation}
	\label{enough}
	\|u(x,\omega)-u^1_h(x,\omega)\|_H\leq (\epsilon+Ch)\|f(x)\|_{L^2(D)}\leq C\epsilon\|f(x)\|_{L^2(D)},
\end{equation}
where we have assumed $h$ is small ($h\leq \epsilon$) in the second inequality. \eqref{enough} implies that when constructing trial space $V_h$ to approximate the solution space of \eqref{targ} for $f(x)\in L^2(D)$, we do not need to consider the high frequency part of $L^2(D)$ that cannot be resolved by the given mesh.
\end{Rem}
With this discretization of $L^2(D)$, we define $P_1$ as the mapping from $R^N$ to $\hat D$,
\begin{equation}
	P_1 v=\sum_{i=1}^Nv_i\Phi_i(x),\quad \text{for}\quad v=[v_1,v_2,\dots, v_N]^T.
	\label{P}
\end{equation}
$P_1$ is an isometry because we have chosen $\Phi_i(x)$ to be orthonormal.

Then we consider discretizing the range of $T_i$ using a set of orthonormal functions in $L^2(\Omega)$. For example, we can use orthonormal polynomials,
\begin{equation}
	\label{disrange}
	\{H_1(\omega), H_2(\omega),\dots, H_M(\omega)\}.
\end{equation}
Then we define the projection operator $P_2$ that maps from $L^2(\Omega)$ to $R^M$,
\begin{equation}
	\label{P2}
	P_2 H(\omega)=(u_1,\dots, u_M)^T,\quad {s.t.} \quad \forall j, \quad \left( H(\omega)-\sum_{i=1}^Mu_iH_i(\omega)\right)\perp H_j(\omega).
\end{equation}

With the discretization of $L^2(D)$ and $L^2(\omega)$ using $P_1$ and $P_2$, $T_i$ is discretized to 
\begin{equation}
	\label{hatti}
	\hat T_i=P_2T_i P_1,
\end{equation}
which is a linear map from $R^N$ to $R^M$ thus is a matrix. $M$ and $N$ should be large to make sure the discretization error introduced in \eqref{discretizeL2D} and \eqref{disrange} is small. 

We want to approximate the range of $\hat T_i$, which means finding a matrix $Q_i$ with orthonormal columns, such that $\|(I-Q_iQ_i^T)\hat T_i\|$ in the operator norm is small. Note that $Q_iQ_i^T$ is the projection operator to the column space of $Q_i$. Once we get $Q_i=(q_i^{ls})$, which is a $M\times k_i$ matrix, we can construct the corresponding stochastic basis $\xi_i^j(\omega)$ approximating the range of $T_i$ based on the operator $P_2$,
\begin{equation*}
	\xi_i^j(\omega)=\sum_{l=1}^M H_l(\omega)q_i^{lj}, \quad j=1,\dots, k_i,
\end{equation*}
where $H_l(\omega)$ are the orthonormal basis functions in $L^2(\Omega)$ \eqref{disrange}. 

We apply the randomized range finding algorithms \cite{halko2011finding,liberty2007randomized} to $\hat T_i$ to construct the $Q_i$. 
\subsection{The Randomized Range Finding Algorithms}
Given a large matrix $A=\hat T_i$ and a tolerance $\epsilon$, we want to construct a matrix $Q$ with orthonormal columns, whose column space can capture the main action of $A$, namely (in operator norm)
\begin{equation}
	\label{range}
	\|(I-QQ^T)A\| \leq \epsilon,
\end{equation}
where $I-QQ^T$ is the projection operator to the orthogonal complement of the column space of $Q$.

The basic idea of the randomized range finding algorithm is using the operator $A$ to act on a random matrix $\Omega$ and extracting the main action of $A$ from the image $A\Omega$. The following lemma~\cite{halko2011finding} can be used {\em a posteriori} to confirm that \eqref{range} holds (with high probability).
\begin{Lem}
	\label{posteri}
Let $B$ be a real $m\times n$ matrix. Fix a positive integer $r$ and a real number $\alpha>0$. Draw an independent family of standard Gaussian vectors
\[\{\omega^{(i)}:i=1,2\dots, r\}.\]
Then
\begin{equation}
	\|B\|\leq \alpha\sqrt{\frac{2}{\pi}}\max_{i=1,\dots, r}\|B\omega^{(i)}\|
	\label{aposteriori}
\end{equation}
except with probability $\alpha^{-r}$.
\end{Lem}

\begin{algorithm}
	\caption{Adaptive Randomized Range Finder}
	\label{rfinder}
	\begin{algorithmic}[1]
		\State Draw standard Gaussian vectors $\omega^{(1)},\dots, \omega^{(r)}$ of length $N$.
		\State For $i=1,2\dots, r$ compute $y^{(i)}=A\omega^{(i)}$.
		\State $j=0$.
		\State $Q^{(0)}=[]$, the $m\times 0$ matrix.
		\While{$\max\{\|y^{(j+1)}\|,\|y^{(j+2)}\|,\dots,\|y^{(j+r)}\|\}>\epsilon/(10\sqrt{2/\pi})$}
		\State $j=j+1$.
		\State Overwrite $y^{(j)}$ by $(I-Q^{(j-1)}(Q^{(j-1)})^T)y^{(j)}$.
		\State $q^{(j)}=y^{(j)}/\|y^{(j)}\|$.
		\State $Q^{(j)}=[Q^{(j-1)} q^{(j)}]$.
		\State Draw a standard Gaussian vector $w^{j+r}$ of length $n$.
		\State $y^{(j+r)}=(I-Q^{(j)}(Q^{(j)})^T)A\omega^{(j+r)}$.
		\For{$i=(j+1),(j+2),\dots,(j+r-1)$,}
		\State Overwrite $y^{i}$ by $y^{(i)}-q^{(j)}\langle q^{(j)},y^{(i)}\rangle$.
		\EndFor
		\EndWhile
		\State $Q=Q^{(j)}$
	\end{algorithmic}
\end{algorithm}

Algorithm \ref{rfinder} \cite{halko2011finding} produces an orthonormal $Q$ such that \eqref{range} holds with probability at least $1-\min\{M,N\}10^{-r}$.  The $Q$ is obtained by Gram-Schmidt orthogonalization in Step $7, 8, 9$ to the image $A\Omega$. Lemma~\ref{posteri} is used in Step $5$ with $B=(I-QQ^T)A$, $\omega^{(i)}=\omega^{j+i}$, $i=1,\dots, r$ to check whether the returned $Q$ makes \eqref{range} hold. The error estimation in this Algorithm is almost free since the computed $(A-Q^TQA)\omega^{j+1}$ can be put back in $Q$ after normalization if \eqref{aposteriori} is not satisfied. Another strategy to extract the main action of $A$ from the image $A\Omega$ is by Singular Value Decomposition as \cite{lin2011fast}. This strategy together with error estimation using Lemma~\ref{range} leads to Algorithm~\ref{localbasis}, which produces an orthonormal $Q$ such that \eqref{range} holds with probability at least $1-2\min\{M,N\}10^{-r}$. In Step $9, 10, 11$, the first several left singular vectors of $W=A\Omega$ are selected to make \eqref{posteri} hold. If all the left singular vectors of $W$ cannot make \eqref{posteri} hold in Step 10, $W$ will be enriched in Step $15$.

Algorithm~\ref{localbasis} is more expensive than Algorithm~\ref{rfinder} because of the SVD process, which might need to be done several times. But based on our numerical tests, Algorithm~\ref{localbasis} in general returns a relatively smaller $Q$ because of the SVD procedure. We actually implement both Algorithms in the HSFEM framework, and their performances are very similar. For clarity we use Algorithm~\ref{localbasis} to demonstrate the implementation and efficiency of our method for the rest of this paper. We remark that the estimate~\eqref{aposteriori} is pessimistic\cite{halko2011finding}, so in practice a small $r$ in the {\em a posteriori } estimation is enough.

\begin{algorithm}
	\caption{Construct An Orthonormal Matrix $Q$ to Approximate the Range of A}
	\label{localbasis}
	\begin{algorithmic}[1]
		\State Draw standard Gaussian vectors $\omega^{(1)},\dots, \omega^{(r)}$ of length $N$.
		\State For $i=1,\dots, r$ compute $y^{(i)}=A\omega^{(i)}$. $Y^{(i)}=y^{(i)}$.
		\State $Q=[]$. 
		\State Draw Gaussian matrix $\Omega$ of size $N\times K$. $W=A\Omega$.
		\While {$\max\{\|Y^{(1)}\|,\|Y^{(2)}\|,\dots,\|Y^{(r)}\|\}>\epsilon/(10\sqrt{2/\pi})$}
		\State Do singular value decomposition to $W$. $W=\sum_{i=1}^K\sigma_iU_i V_i^T$ with $\sigma_1\geq \sigma_2\geq \dots\geq \sigma_K$.
		\State $\gamma=1$.
		\While {$\gamma\leq K$}

		\State For $i=1,\dots, r$ compute $Y^{(i)}=Y^{(i)}-U_\gamma U_\gamma^TY^{(i)}$.
\If {$\max\{\|Y^{(1)}\|,\|Y^{(2)}\|,\dots,\|Y^{(r)}\|\}<\epsilon/(10\sqrt{2/\pi})$}

		\State Let $Q=[U_1,U_2,\dots, U_\gamma$].
		\Return
		\EndIf
		\State $\gamma=\gamma+1$.
		\EndWhile
		\State Draw Gaussian $\hat \Omega$ of size $N\times K$. $W=[W,A\hat \Omega]$.
		\State $K=2K$.
		\State For $i=1,\dots, r$, $Y^{(i)}=y^{(i)}$.
		\EndWhile
	\end{algorithmic}
\end{algorithm}

\subsection{Implementation Details of the Sampling Method}
\label{samplingdetail}
We apply Algorithm~\ref{localbasis} to the operator $\hat T_i$, $i=1,\dots, n$ to construct the local stochastic basis. Note that $T_i$ is defined as a solution operator to equation~\eqref{targ}, so we do not have direct access to each entry of $\hat T_i$. However, for any $v=(v_1,\dots, v_N)^T\in R^N$, based on \eqref{P}, \eqref{hatti} and \eqref{Ti}, we can compute $\hat T_iv$ by solving equation~\eqref{targ} using the forcing function
\[f(x)=\sum_{i=1}^n v_i\Phi_i(x).\]
Denote the solution as $u(x,\omega)$, then we have
\begin{equation}
\hat T_iv=P_2\left(u(x_i,\omega)-E[u(x_i,\omega)]\right),
\label{mvmultiply}
\end{equation}
where $P_2$ is defined in \eqref{P2}, and this means we have access to matrix-vector multiplication of $\hat T_i$. 

Fortunately, Algorithm \ref{localbasis} does not require access to each entry of $A=\hat T_i$. It only involves using the operator to act on a random matrix $\hat\Omega$, which we have access to based on \eqref{mvmultiply}, and doing singular value decomposition to the image. We never need to implement the projection operator $P_2$ in practice. Instead we view the columns of the image $W=A\Omega$ as $L^2(\Omega)$ vectors and do Singular Value Decomposition in Step $6$ and Gram-Schmidt orthogonalization in  Step $9$ according to the $L^2(\Omega)$ inner product. 

To avoid doing SVD more than once, we choose the initial $K$ in Step $4$ large enough, such that the main action of $A$ can be captured in $W$ in Step $4$, i.e., $W$ does not need to be enriched in Step $15$. The implementation details of this sampling method are given below:
\begin{enumerate}
	\item Generate a Gaussian matrix $\Omega$ of size $N\times r$, and get the corresponding $r$ forcing functions based on the operator $P_1$,
		\[f^p(x)=\sum_{q=1}^N\Omega(q,p)\Phi_q(x),\quad p=1,\dots, r.\]
		Solve Equation \eqref{targ} using forcing $f^p(x)$, $p=1,\dots, r$, and denote the solutions as $u^{p}(x,\omega)$. Let
		\[Y^p_i(\omega)=u^p(x_i,\omega)-E[u^p(x_i,\omega)].\]
	\item Generate a Gaussian matrix $\Omega$ of size $N\times K$, and get the corresponding $K$ forcing functions based on the operator $P_1$,
		\[f_p(x)=\sum_{q=1}^N\Omega(q,p)\Phi_q(x),\quad p=1,\dots, K.\]
		This means we are sampling each operator $\hat T_i$ using the same random matrix $\Omega$. Solve equation \eqref{targ} with forcing $f_p(x)$, $p=1,\dots, K$, and denote the solutions as $u_p(x,\omega)$. Let 
		\begin{equation*}
			y_i^{(p)}(\omega)=u_p(x_i,\omega)-E[u_p(x_i,\omega)], \quad i=1,\dots, N, \quad p=1,\dots, K.
		\end{equation*}
	\item For each node point $x_i$, compute the $K\times K$ matrix $C_i$ as
		\begin{equation*}
			C_i(p,q)=\int_\Omega y_i^{(p)}(\omega)y_i^{(q)}(\omega)\mathrm{d}P,
		\end{equation*}
		which is symmetric and positive semi-definite.
	\item For each $i$, compute the eigen-decomposition of $C_i$, which is
		\begin{equation*}
			C_i=\sum_j^K \lambda_i^{j}V_i^j(V_i^j)^T,
		\end{equation*}
		where $V_i^j$ are $K\times 1$ matrix. $C_i$ can be decomposed efficiently if $K$ is small.
\item Denote the $p$-th entry of $V_i^j$ as $V_i^j(p)$, then we get the stochastic basis functions for each node $x_i$. 	
		\begin{equation*}
			\xi_i^j(\omega)=\sqrt{\frac{1}{\lambda_i^j}}\sum_{l=1}^K V_i^{j}(p)y_i^{(p)}(\omega),\quad j=1,\dots, K.
		\end{equation*}

	\item For each $i$, find the smallest $\gamma=k_i$ such that for all $p=1,\dots, r$,
		\begin{equation*}
			\|Y_i^p(\omega)-\sum_{j=1}^\gamma\int_\Omega Y_i^p(\omega)\xi_i^j(\omega)\mathrm{d}P\xi_i^j(\omega)\|\leq \epsilon/(10\sqrt{2/\pi}).
\end{equation*}
			Then the local stochastic basis functions $\xi_i^j(\omega)$ at node $x_i$ are
			\begin{equation*}
				\xi_i^j(\omega),\quad j=1,\dots, k_i.
			\end{equation*}
		They have mean $0$, variance $1$, and are orthogonal to each other.
\end{enumerate}

\begin{Rem} 
The constructed local stochastic basis functions are not exactly the first several left singular vectors of $T_i$, but to obtain certain accuracy in $L^2(\Omega)$, the number of returned local stochastic basis functions is close to the optimal $k_i$ if the singular values of $T_i$ decay fast \cite{halko2011finding}. 
\end{Rem}
\begin{Rem}
	Following the same argument as Lemma~\ref{nouse}, we know that there is high probability that the trial space $V_h$ constructed using the sampling method has the following approximation property,
	\begin{equation}
		\inf_{v(x,\omega)\in V_h}\|u(x,\omega)-v(x,\omega)\|_{L^2(D\times\Omega)}\leq C\epsilon.
		\label{trialspaceproperty}
	\end{equation}
	Note that according to \eqref{optimality}, the HSFEM searches the best approximation of the solution within the trial space $V_h$ in $L^2(H^1_0(D),\Omega)$, not $L^2(D\times\Omega)$. So the approximation property~\eqref{trialspaceproperty} cannot guarantee convergence of the numerical solution in $L^2(D\times\Omega)$. Convergence analysis of the HSFEM framework will be given in our future work. Numerical results in section~\ref{examples} suggest that the local stochastic basis constructed using the sampling method works well within the HSFEM framework.
\end{Rem}

\section{Numerical Implementation of the HSFEM}
\label{implementationsection}
In this section, we address several issues concerning the implementation of the HSFEM. We first summarize the outline of the whole algorithm and estimate the main computational cost. Methods to reduce the offline computational cost and an online error estimation and correction procedure are given. Then we discuss two ways of discretization in the stochastic direction. The HSFEM is parallel in nature and can be easily implemented on a parallel machine to attain more computational savings.

\subsection{Outline of the Whole Method and Main Computational Cost}
The HSFEM involves two stages: the offline stage and the online stage. In the offline stage, we construct the local stochastic basis by sampling the operator using randomly generated forcing functions for a number of times and form the stiffness matrix. In the online stage, we solve equation~\eqref{targ} efficiently for multiple forcing functions using the coupling basis constructed from the offline stage. 

The {\bf offline stage} involves the following procedures:
\begin{itemize}
	\item Construct the local stochastic basis functions using Algorithm \ref{localbasis}, and denote them as 
		\[\xi_i^j(\omega), \quad i=1,\dots, n, \quad j=0,\dots, k_i.\]
		The implementation details are given in the previous section.

	\item Construct the coupling basis $\{\phi_i(x)\xi_i^j(\omega)\}$ and compute the stiffness matrix $SM$. 
			\begin{itemize}	
			\item Let $S$ be the number of coupling basis functions $S=\sum_{i=1}^nk_i+n$, then $SM$ is of size $S\times S$. 
		
			\item Let $R$ be the relabeling function that maps each pair $(i,j), i=1,\dots, n,\ j=0,\dots, k_i$ to the global index of the coupling basis function $\phi_i(x)\xi_i^j(\omega)$: $R(i,j)=\sum_{l=1}^{i-1}(k_l+1)+j+1$.
	\item Compute the stiffness matrix as
		\begin{equation}
			SM(R(i_1,j_1),R(i_2,j_2))=\int_\Omega\int_D\xi_{i_1}^{j_1}(\omega)\nabla \phi_{i_1}(x)^Ta(x,\omega)\xi_{i_2}^{j_2}(\omega)\nabla \phi_{i_2}(x)\mathrm{d}x\mathrm{d}P.
			\label{smatrix}
		\end{equation}

		Since $\phi_i(x)$ has only local support, the stiffness matrix is sparse.
	\end{itemize}
\end{itemize}

The offline stage could be quite expensive because in constructing the local stochastic basis, we need to solve equation~\eqref{targ} with randomly chosen forcing functions using traditional methods for $K$ times and do singular value decomposition to $\hat T_i\hat \Omega$, which has $K$ columns. However, if the singular values of operator $T_i$~\eqref{Ti} decay fast, a small $K$ is enough for the construction of the local stochastic basis. Our numerical results in section~\ref{examples} demonstrate that this indeed holds for elliptic operators with several types of stochastic input. For $K$ small, the {\bf main computational cost} in the offline stage comes from sampling the stochastic operator $K$ times. It is of order
\begin{equation}
	\label{offcost}
	O(KMn^2),
\end{equation}
where $M$ is the number of sampling points (we use the SC or MC method in our numerical examples in section~\ref{examples}). On each sampling point, we need to solve a deterministic elliptic equation using $\phi_i(x)$. We have assumed the computational cost of solving a sparse $n\times n$ linear system is $O(n^2)$ in \eqref{offcost}.

In the {\bf online stage} we consider solving the equation with $F$ different forcing functions: 
\begin{itemize}
	\item For each forcing function $f(x)$, construct the load  vector $b$, which is of size $S\times 1$.
		\begin{equation}
			\label{loadvector}
			b(R(i,j))=\int_\Omega\int_D \xi_i^j(\omega)\phi_i(x)f(x)\mathrm{d}x\mathrm{d}P.
		\end{equation}
		The local stochastic basis constructed using the sampling method in the previous section has mean $0$, so only the $R(i,0)$th entries, $i=1,\dots, n$, of the load vector $b$ are non-zero.
	\item Solve $c_i^j$ from the linear system
		\begin{equation}
		\label{system}
		SM\times c=b.
		\end{equation}
		Then the numerical solution is
		\begin{equation}
			u_h(x,\omega)=\sum_{i=1}^n\sum_{j=1}^{k_i}c_i^j\phi_i(x)\xi_i^j(\omega).
			\label{finalsolu}
		\end{equation}
	\item Compute the quantities of interest based on the numerical solution $u_h(x,\omega)$. 
		
		Since the local stochastic basis functions at each node are orthonormal, we can compute the mean and variance of the solution efficiently without assembling the local stochastic basis functions. 
		\[E[u(x_i,\omega)]=c_i^0, \quad \sigma^2[u(x_i,\omega)]=\sum_{j=1}^{k_i}(c_i^j)^2.\]
\end{itemize}

In the online stage, the main computation cost comes from solving the linear system~\eqref{system}. If the average number of stochastic basis functions for each node, $k$, is small, then the linear system~\eqref{system} is small ($kn\times kn$) and sparse, thus can be solved efficiently. The total online computational cost is of order
\begin{equation}
	\label{oncost}
	O(Fk^2n^2),
\end{equation}
in which we again assume sparse linear system of size $kn\times kn$ can be solved in $O(k^2n^2)$ operations. 

Combining \eqref{offcost} and \eqref{oncost}, we get that the total computational cost of the HSFEM is
\begin{equation}
	\label{hsfem}
	C_1KMn^2+C_2Fk^2n^2,
\end{equation}
where $C_1$ and $C_2$ are two generic constants. On the other hand, if the equation is solved using SC methods for $F$ times, the computational cost is
\begin{equation}
	C_1FMn^2.
	\label{oldmethod}
\end{equation}

The computational cost of SC method and the HSFEM in the multi-query setting is illustrated in Figure~\ref{fig:comparecost}. The $T$ in the vertical axis is the offline computational cost of the HSFEM, $T\approx C_1KMn^2$. The $F^*$ in the horizontal axis is the critical number of queries when the cost of SC and the HSFEM are equal, $F^*\approx C_1KMn^2/(C_1Mn^2-C_2k^2n^2)$. The slope of the blue line $k^2n^2$ is much less than that of the red line $Mn^2$ when the elliptic operator enjoys the Operator-Sparsity, which means $k$ is small.
\begin{figure}[htpb]
	\begin{center}
		\includegraphics[width=0.5\textwidth]{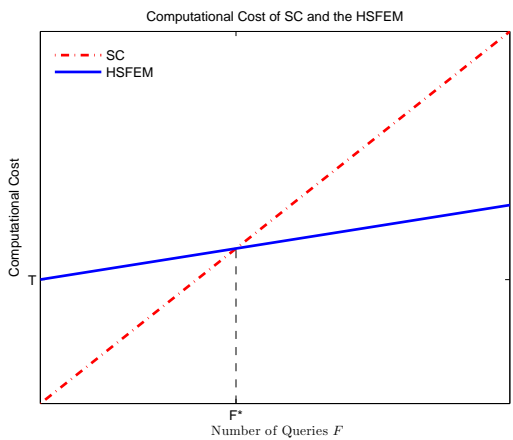}
	\end{center}
	\caption{Computational cost of SC methods and the HSFEM in the multi-query setting.}
	\label{fig:comparecost}
\end{figure}

Based on the computational cost \eqref{hsfem} and \eqref{oldmethod}, we can see that:
\begin{itemize}
	\item If the stochastic operator enjoys the Operator-Sparsity, which means $k$ is small, the online computation can be very efficient since the linear system we solve is sparse and small.
	\item Our method can achieve computational savings only if $F>F^*$. 
		
		If the equation~\eqref{targ} only needs to be solved once, then our method cannot bring in any computational savings because of the offline computation. However, if we need to solve it multiple times with different forcing functions, then our method can be very efficient.
\end{itemize}
\subsection{Methods to Reduce the Offline Computational Cost and An Online Error Estimation and Correction Procedure}
\label{correction}
The offline computational cost is expensive since it requires solving equation \eqref{targ} $K$ times with randomly chosen forcing functions using a traditional method. We discuss methods to reduce the offline computational cost and an online error estimation and correction procedure in this subsection.

The following strategies can be considered to reduce the offline computational cost:
\begin{itemize}
\item Construct the local stochastic basis on a coarser mesh.

	We can use a coarser mesh in the offline stage to reduce the $n$ in \eqref{offcost}. To be specific, let $x^i, i=1,\dots, n_c$ be a coarse mesh and $\phi^i(x), i=1,\dots, n_c$ be the corresponding coarse mesh piecewise linear basis. We use $\phi^i(x)$ to discretize the equation in the offline stage and the sampling solutions are represented by $\phi^i(x)$. We restrict them to the fine mesh nodes $x_i,i=1,\dots, n$, and then follow the implementation details in section~\ref{samplingdetail} to construct the local stochastic basis functions associated with $x_i$. Similar strategy has been employed to reduce computational cost in \cite{cheng2013data, doostan2007stochastic}.    

\item Choose a relatively large $\epsilon$ in the randomized range finding algorithm.

	In certain cases, the singular values of operator $T_i$ do not decay very fast, and a large number of local stochastic basis functions are required to obtain accuracy $\epsilon$, thus the required $K$ in the offline stage is large. To reduce the offline computational cost, we can choose a relative larger $\epsilon$, thus reduce the $K$ in the computational cost \eqref{offcost}. Note that the constructed local stochastic basis using the larger $\epsilon$ is not as accurate, and the online numerical solutions may have large error. 

	\item Choose a small $K$ in the randomized range finding algorithm.

		It will be shown in section~\ref{examples} that the solution space to \eqref{targ} has spatially heterogeneous stochastic structure, and the required number of local stochastic basis functions $k_i$ are different for different nodes $x_i$. If for certain (small) region of the domain, the required number of local stochastic basis functions $k_i$ is very large, then a large $K$ is required in the offline stage. However, for other regions of the domain, a smaller $K$ may be enough in Algorithm~\ref{localbasis}. To reduce the computational cost, we can choose a relatively small $K$. And in the Step 8-14 of Algorithm~\ref{localbasis}, if all the left singular vectors of $A\Omega$ cannot make \eqref{posteri} hold, we simply return the Algorithm with $Q=[U_1,\dots, U_K]$. 
\end{itemize}

The strategies mentioned above can reduce the offline computational cost to some degree, but the resulting local stochastic basis is not as accurate, and the corresponding online numerical solution may have large error. Here we introduce a procedure of online error estimation and correction \cite{cheng2013data} through Monte Carlo sampling. We want to emphasize that this {\em a posteriori } error estimation and correction procedure can be incorporated into the HSFEM framework to get more faithful numerical results even if the above strategies to reduce offline computational cost are not taken.

In many practical UQ problems, we care about the statistical quantity of the solution $u(x,\omega)$, which we denote by $E[g(u(x,\omega))]$, where $g$ is a functional defined on realizations of the stochastic solution. Denote $\omega^i, i=1,\dots, N_{MC}$ as $N_{MC}$ Monte Carlo sampling points, then $E[g(u(x,\omega))]$ can be approximated by
\begin{equation}
	\label{mcestimator}
	I_{N_{MC}}[g(u)]=\frac{1}{N_{MC}}\sum_{n=1}^{N_{MC}}g(u(x,\omega^n)).
\end{equation}
We denote the error of the Monte Carlo estimator \eqref{mcestimator} as
\[\epsilon_{g(u)}(N_{MC})=E[g(u)]-I_{N_{MC}}[g(u)].\]
Then $\epsilon_g(N_{MC})$ is a random variable, and according to the central limit theorem, the root-mean-square error of the Monte Carlo estimator \eqref{mcestimator} decays with $N_{MC}$ like
\begin{equation*}
	\sqrt{E[\epsilon_{g(u)}^2(N_{MC})]}= \frac{\sigma[g(u)]}{\sqrt{N_{MC}}},
\end{equation*}
where $\sigma[g(u)]$ is the standard deviation (SD) of $g(u)$. Note that the decay rate of the error $\frac{1}{\sqrt{N_{MC}}}$ is low, and this limits the application of MC methods. One way to accelerate the convergence is reducing the variance of $g(u)$. Using the HSFEM numerical solution $u_h(x,\omega)$, we can divide $E[g(u)]$ as
\begin{align}
	E[g(u(x,\omega))]&=E[g(u_h(x,\omega))]+E[g(u(x,\omega))-g(u_h(x,\omega))]\\
	&\approx E[g(u_h)]+\sum_{k=1}^{N_{MC}}\frac{1}{N_{MC}}(g(u(x,\omega^k))-g(u_h(x,\omega^k))).\label{rmc}
\end{align}
The first part in \eqref{rmc} can be efficiently computed using the HSFEM in the online stage. As to the second part, which is a Monte Carlo estimator, we have that its root-mean-square error decays as
\begin{equation*}
	\sqrt{E[\epsilon^2_{g(u)-g(u_h)}(N_{MC})]}=\frac{\sigma[g(u)-g(u_h)]}{\sqrt{N_{MC}}},
\end{equation*}
where $\sigma[g(u)-g(u_h)]$ is the SD of $g(u)-g(u_h)$.

We make a {\bf key assumption} here that the numerical solution $u_h$ using the HSFEM, even though may be not accurate, can still capture the main stochastic part of $u(x,\omega)$. Then $g(u)-g(u_h)$ has significantly smaller variance than $g(u)$, and much fewer MC samples in approximating $E[g(u)-g(u_h)]$ in \eqref{rmc} than in \eqref{mcestimator} is enough to obtain high accuracy. From this point of view, the HSFEM can be viewed as a tool for {\bf variance reduction}. We denote $\tau(\omega)=g(u)-g(u_h)$, which has relatively small variance, and set a threshold $\varepsilon$. Then the error correction procedure includes: 
\begin{enumerate}
	\item  Generate $N_{MC}$ MC samples, $\omega^k$, $k=1,\dots, N_{MC}$.
	\item  Denote $\bar \tau$, and $\tilde \tau$ as
		\begin{equation}
			\label{mcmcestimate}
			\bar\tau=\frac{1}{N_{MC}}\sum_{k=1}^{N_{MC}}\tau(\omega^k),\quad \tilde \tau=\sqrt{\frac{1}{N_{MC}(N_{MC}-1)}\sum_{k=1}^{N_{MC}}(\tau(\omega^k)-\bar \tau)^2}.
		\end{equation}
		Then an approximate $95\%$-level confidence interval of $E[\tau(\omega)]$ is given by $[\bar \tau-2\tilde \tau, \bar \tau+2\tilde \tau]$.
	\item Consider the following three cases:
		\begin{itemize}
			\item $(|\bar\tau|+2|\tilde\tau|)/|E[g(u_h)]|\leq \varepsilon$, which means the numerical result using the HSFEM is accurate. Then we use $E[g(u_h)]$ as a faithful approximation of $E[g(u)]$.
			\item If the first case does not hold, but $\tilde \tau<\varepsilon|E[g(u_h)]+\bar \tau|$, which means the error in the estimate of $E[\tau]$ is negligible. Then we use $E[g(u_h)]+\bar \tau$ as a faithful approximation of $E[g(u)]$.
			\item If neither of the above two cases holds, which means the number of MC samples $N_{MC}$ is not large enough, then we double the sampling number $N_{MC}$ and go back to step 1.
		\end{itemize}
\end{enumerate}
\subsection{Discretization in the Stochastic Direction}
\label{disOmega}
In the spatial direction, the domain $D$ is discretized using a mesh with nodes $x_i, i=1,\dots, n$, and the solution is represented using the standard piecewise linear basis, $\phi_i(x)$, $i=1,\dots, n$. While in the stochastic direction, there are basically two ways of discretization:
\begin{itemize}
	\item Represent the local stochastic basis and numerical solution using a set of orthonormal polynomials.	
		
		If in the offline stage we use PC methods to solve equation~\eqref{targ}, or if we use SC methods to solve the equation and interpolate the solution using orthonormal polynomials, then the local stochastic basis and numerical solutions will be represented in this way.

	\item Discretize $(\Omega,P)$ using sampling points $\theta^1,\theta^2,\dots, \theta^M$ with weights $w^1,w^2,\dots, w^M$.

		If in the offline stage we use SC methods or MC methods to solve equation~\eqref{targ}, then the numerical solutions and local stochastic basis functions are represented in this way. In MC methods $w^i=\frac{1}{M}$, while in SC methods $\theta^i$ and $w^i$ are determined by the underlying probability measure. 
\end{itemize}

In our numerical examples, we use the second way to discretize the problem in the stochastic direction, i.e., in the offline stage we solve the equation on sampling points $\theta^i, i=1,\dots, M$ and the local stochastic basis functions are stored as length-$M$ vectors. In computing the stiffness matrix~\eqref{smatrix} and load vector~\eqref{loadvector}, we use the same set of sampling points $\theta^i$ and weights $w^i$ for numerical integration.

The Smolyak sparse grid quadrature \cite{nobile2008sparse, smolyak1963quadrature} can alleviate the \emph{curse of dimensionality} to some degree when the dimension of the stochastic input is not very high and the solution has high regularity in the stochastic direction. In our numerical examples we use the Smolyak sparse grid to generate $\theta^i$ and $w^i$ when the dimension of the stochastic input is not very high. Otherwise we use Monte Carlo method to generate $\theta^i$ and set the weights $w^i=\frac{1}{M}$, for $i=1,\dots, M$.

Under this discretization, our final numerical solution $u_h(x,\omega)$~\eqref{finalsolu} is also defined on the sampling points $\theta^i$. If $\theta$ and $w$ are chosen to be the sparse grid collocation points, we can recover $u(x,\omega)$ for any $\omega\in\Omega$ using polynomial interpolation.\cite{nobile2008sparse} When MC sampling points are used, we can {\bf only} compute the statistical quantities of the solutions based on their values on these sampling points.

\subsection{Parallelization of the HSFEM}
The offline computation cost of the HSFEM can be expensive if a large $K$ is chosen, but the HSFEM is parallel in nature and can be easily implemented on a parallel machine. To be specific:
\begin{itemize}
	\item The offline stage involves solving equation~\eqref{targ} for a number of times using randomly generated forcing functions. These forcing functions are independent, thus this process can be parallelized.
	\item In the construction of local stochastic basis, we need to do singular value decomposition on each node point. And these decompositions are independent, thus can be parallelized.
	\item In the online stage, after solving equation~\eqref{system} and getting the coefficients $c_i^j$, we need to assemble the coupling basis functions $\phi_i(x)\xi_i^j(\omega)$ to recover the numerical solution $u_h(x,\omega)$~\eqref{finalsolu}. The assembling is independent on each node point thus can be parallelized.
\end{itemize}
\section{Numerical Results}
\label{examples}
In this section, we present numerical results to demonstrate that:
		\begin{itemize}
			\item The solution space to stochastic elliptic equation has spatially heterogeneous stochastic structure, and this can be recognized by the HSFEM in the offline stage.
			\item Several elliptic operators with high dimensional stochastic input enjoys the Operator-Sparsity. 
			\item The local stochastic basis constructed using the sampling method works well within the HSFEM framework. And the HSFEM can take advantage of the sparsity of the operator to attain significant computational savings while maintaining high accuracy in the online stage.
		\end{itemize}

To quantify the accuracy of our numerical solution, we need to compare it  with the exact solution. However in most cases it is impossible to construct the exact solution analytically, so we instead choose a suitable numerical solution as a reference. In our numerical implementation, we discretize the domain $D$ using a piecewise linear basis $\phi_i(x)$, and discretize $(\Omega, P)$ using sampling points $\theta^i$ with weights $w^i$. Then the error in our numerical solution $u_h(x,\omega)$ can be divided into two parts as
\begin{equation}
	u_h(x,\omega)-u(x,\omega)=[u_{d}(x,\omega)-u(x,\omega)]+[u_h(x,\omega)-u_{d}(x,\omega)],
	\label{decomposeerror}
\end{equation}
where $u_{d}(x,\omega)$ is the numerical solution to equation~\eqref{targ} using the SC (or MC) method. $u_d(x,\omega)$ is obtained by solving equation~\eqref{targ} on the discrete sampling points $\theta^i, i=1,\dots, M$, using the spatial basis $\phi_i(x), i=1,\dots, n$. The first part of the error in \eqref{decomposeerror} is the {\bf discretization error}, which comes from the discretization of $D$ and $\Omega$. The second part of \eqref{decomposeerror} comes from our selection of finite local stochastic basis functions from $L^2(\Omega)$. For the purpose of evaluating the performance of the HSFEM, we ignore the discretization error and only consider the second part of the error, namely, we use $u_d(x,\omega)$ as the reference solution. In the rest of this section we denote $u_{d}(x,\omega)$ as $u(x,\omega)$ for simplicity.

In each of our numerical examples in this section we compute the average number of the local stochastic basis functions constructed in the offline stage, which is
		\begin{equation*}
			k=\sum_{i=1}^n\frac{k_i}{n}.\label{ANSB}
		\end{equation*}
		If $k$ is small, and the numerical error using the HSFEM is small, then a small set of coupling basis $\{\phi_i(x)\xi_i^j(\omega)\}$ can approximate the solution space well. Thus the stochastic elliptic operator that we consider enjoys the Operator-Sparsity. Besides, the linear system we solve in the online stage is of size $n(k+1)\times n(k+1)$, so smaller $k$ also implies better efficiency of the HSFEM.

We also compute the error in Expectation (mean) and Standard Deviation (SD) of the solution, which are the two primary quantities of interest in Uncertainty Quantification:
		\begin{equation*}
			\begin{split}
			E[u]-E[u_h]&=\sum_{i=1}^Mw^i( u(x,\theta^i)-u_h(x,\theta^i) ).\\
		\sigma[u]-\sigma[u_h]&=(\sum_{i=1}^Mw^i(u(x,\theta^i)-E[u(x,\omega)])^2)^{1/2}-(\sum_{i=1}^Mw^i(u_h(x,\theta^i)-E[u_h(x,\omega)])^2)^{1/2}.
	\end{split}
	\end{equation*}

We compute the relative $L^2(D\times\Omega)$ error in the stochastic part of the solution,
		\begin{equation}
			E_{HSFEM}=\frac{ \left(\sum_{i=1}^Mw^i\|u(x,\theta^i)-u_h(x,\theta^i)\|_{L^2(D)}^2\right)^{\frac{1}{2}}}
			{\left(\sum_{i=1}^Mw^i\|u(x,\theta^i)-\bar{u}(x)\|_{L^2(D)}^2\right)^{\frac{1}{2}}
			}.
			\label{rl2}
		\end{equation}
		Note that in the denominator is the $L^2(D\times\Omega)$ norm of $u(x,\omega)-\bar u(x)$, not $u(x,\omega)$, so $E_{HSFEM}$ measures the capacity of the HSFEM in capturing the stochastic part of the solution.
\subsection{Comparison with the KL Expansion}
\label{subtlecompare}
The key difference of the HSFEM framework from traditional methods is that different stochastic basis functions are used in different regions of the domain to approximate the stochastic part of the solution. By doing so we allow for spatially heterogeneous stochastic structure of the solution space and expect to reduce the total degrees of freedom and computational cost in the online stage. 

To demonstrate the advantage of approximating the stochastic behavior of the solution locally, we compare the accuracy of the HSFEM numerical solution with the truncated KL expansion of the exact solution, which, as we introduced in section~\ref{klsection}, is the {\bf optimal approximation} to the stochastic part of the solution in $L^2(D\times\Omega)$ when global stochastic basis functions are used.

Recall that the numerical solution using the HSFEM $u_h(x,\omega)$ is represented as
\begin{equation*}
	\label{hsfemsolu}
	u_h(x,\omega)=\sum_{i=1}^n\sum_{j=0}^{k_i}c_i^j\phi_i(x)\xi_i^j(\omega),
\end{equation*}
in which a total number of
\begin{equation*}
	S=\sum_{i=1}^n (1+k_i)
\end{equation*}
basis functions are used. On the other hand, consider the $\tilde k$-term KL expansion of $u(x,\omega)$,
\begin{equation}
	\label{klapp}
	u_{\tilde k}(x,\omega)=\bar u(x)+\sum_{j=1}^{\tilde k}\sqrt{\lambda_j}\psi_j(x)\xi_j(\omega)=\sum_i^n\bar u(x_i)\phi_i(x)+\sum_{j=1}^{\tilde k}\sum_{i=1}^n\sqrt{\lambda_j}\psi_j(x_i)\phi_i(x)\xi_j(\omega),
\end{equation}
in which a total number of 
\begin{equation*}
	n+\tilde kn
\end{equation*}
basis functions are used. To make a fair comparison between the HSFEM numerical solution and the truncated KL expansion, we should keep the degrees of freedom the same, namely, 
\begin{equation*}
	S=n+\tilde kn,
\end{equation*}
which means we choose $\tilde k=k=\frac{\sum_i^nk_i}{n}$ in the truncated KL expansion \eqref{klapp} to make the comparison.

Note that in both $u_h(x,\omega)$ and $u_k(x,\omega)$, $k$ (on average) stochastic basis functions are used to approximate the stochastic part of the solution on each node $x_i$. The difference is that the local stochastic basis functions $u_h(x,\omega)$ uses are different on each node point, while the stochastic basis functions in $u_k(x,\omega)$ are the same on the whole domain. We define the relative KL truncation error as
\begin{equation}
	\label{rte}
	E_{KL}=\frac{\|u(x,\omega)-u_k(x,\omega)\|_{L^2(D\times\Omega)}}{\|u(x,\omega)-\bar u(x,\omega)\|_{L^2(D\times\Omega)}}.
\end{equation}
We compare $E_{KL}$ with the numerical error $E_{HSFEM}$ \eqref{rl2} for each numerical example in this section. We emphasize that this is an unfair comparison since the stochastic basis in $u_k(x,\omega)$ is adapted to the solution with a specific forcing and does not necessarily approximate another solution well, while the stochastic basis in $u_h(x,\omega)$ is constructed to approximate the whole solution space for all $f(x)\in L^2(D)$. 

So if for some forcing function $f(x)$, $E_{HSFEM}\leq E_{KL}$ (or they are of the same order considering the unfairness of this comparison), then it reflects the heterogeneous stochastic structure of the solution, and the advantage of approximating the stochastic behavior of the solution locally.
\subsection{A Thoroughly Studied 1D Model Problem}
In this subsection we consider the following 1D elliptic SPDE:
\begin{equation} 
	\begin{cases}
		-\frac{\partial}{\partial x}(a(x,\omega)\frac{\partial}{\partial x} u(x,\omega))=f(x),\quad x\in D=(0,1), \quad \omega \in \Omega,\\
		u(0,\omega)=0,\quad u(1,\omega)=0.
	\end{cases}
	\label{onedequation}
\end{equation} 

We assume the KL expansion of $\log a(x,\omega)$ is given by
\begin{equation}
	\label{1dcoef}
	\log a(x,\omega)=\sum_{i=1}^{m}\cos(2\pi i x)\omega_i,
\end{equation}
where $\omega_i$, $i=1,\dots, m$, are independent random variables, 
\begin{equation*}
	\omega_i\sim \mathcal{U}(-1/2,1/2).
\end{equation*}

Note that in the KL expansion of $\log a(x,\omega)$, the first $m$ singular values are equal, which means the $m$ random variables $\omega_i$ contribute equally to $\log a(x,\omega)$ in the $L^2(D\times \Omega)$ sense and none of them is negligible. So this problem has genuine high stochastic dimension for large $m$.
\subsubsection{The Linear Compact Operator $T_i$}
\label{compactsection}
Recall that $T_i$ \eqref{Ti} is a compact linear operator mapping from the forcing $f(x)\in L^2(D)$ to the stochastic part of the solution at $x_i$, $u(x_i,\omega)-\bar u(x_i)\in L^2(\Omega)$. $T_i$ can be discretized to a matrix $\hat T_i$ of size $M\times N$ through \eqref{hatti}. Since $T_i$ is defined as a solution operator, we do not have direct access to each entry of $\hat T_i$, but we have access to matrix-vector multiplication of $\hat T_i$ based on \eqref{mvmultiply}. In this subsection, we use $\hat T_i$ to act on the standard basis $e_k, k=1,\dots, N$, of $R^N$, by doing which we explicitly get all the columns of $\hat T_i$. We use this explicit $T_i$ to investigate the 1D stochastic elliptic operator.

In the spatial direction we discretize the domain $[0,1]$ using uniform mesh of size $h=1/256$. In the stochastic direction, we discretize $(\Omega, P)$ using $(\theta^i, w^i)$, which are the 4th order Smolyak sparse grid collocation points. The orthonormal basis functions to discretize $L^2(D)$~\eqref{basis} are chosen to be
\begin{equation}
	\{1,\dots, 2\sin(2\pi lx), 2\cos(2\pi lx)\}
	\label{discretel2}
\end{equation}
with $l=128$, which contains all the Fourier modes that can be resolved by the given mesh.

We choose $m=20$ in \eqref{1dcoef}, and explicitly compute the $T_i$ with $x_i=1/2$. The decay of the singular values of $T_i$, $\sigma_i^k$, is plotted in Figure ~\ref{fig:Ti_decay}. From this figure, we can see that even though the stochastic input $\omega$ has high dimension, $\sigma_i^k$ still decays exponentially fast. The fast-decay of the singular values of $T_i$ reveals the Operator-Sparsity of this 1D stochastic operator. It also justifies the sampling method to construct the local stochastic basis, since the performance of the randomized range finding algorithm depends on the decay rate of the singular values of the operator.\cite{halko2011finding}
\begin{figure}[htpb]
	\begin{center}
		\includegraphics[width=0.5\textwidth]{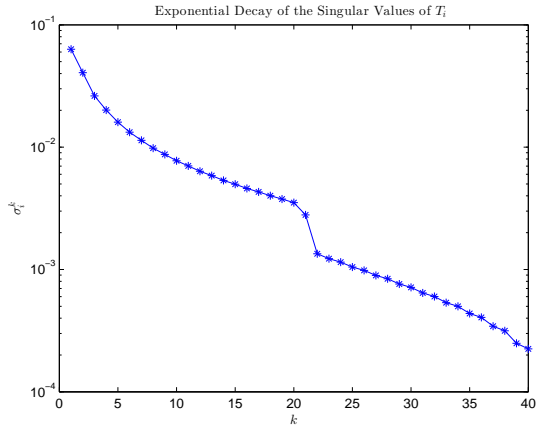}
	\end{center}
	\caption{Exponential decay of the singular values of $T_i$.}
	\label{fig:Ti_decay}
\end{figure}

The number of required terms to make the truncation error of $T_i$ less than $\epsilon$ is 
\[k_i=\inf\{k: \sigma_i^k\leq \epsilon\}.\]
We choose $m$ ranging from $10$ to $30$ in \eqref{1dcoef}, and for each $m$, we compute the number of required terms $k_i$ (with $x_i=1/2$) to make the truncation error of $T_i$ less than $\epsilon=2\times 10^{-3}$. The dependence of $k_i$ on $m$ is plotted in Figure~\ref{fig:k_growth}. From this figure we can see that $k_i$ seems to grow unboundedly (about linearly) with the stochastic dimension. This does not contradict our analysis in section~\ref{interpretation}, since as $m$ increases, the ellipticity of the stochastic operator deteriorates, i.e., $\lambda_{\min}$ decreases and $\lambda_{\max}$ increases. To confirm the analysis in section~\ref{interpretation}, we consider a normalized model with the coefficient $\tilde a(x,\omega)$ given by
\begin{equation*}
\log \tilde a(x,\omega)=\sum_{i=1}^m\frac{20}{m}\cos(2\pi ix)\omega_i,
\end{equation*}
where $\omega_i, i=1,\dots,m$ are independent and $\omega_i\sim\mathcal{U}(-1/2,1/2)$. We can see that for all $m\geq 1$,
\begin{equation*}
	\tilde a(x,\omega)\in [e^{-10}, e^{10}].
\end{equation*}

For this normalized model, we use the same discretization as the unnormalized model and compute the number of required terms $k_i$ to make the truncation error of $T_i$ less than $\epsilon=2\times 10^{-3}$. The dependence of $k_i$ on $m$ is plotted in Figure~\ref{fig:k_no_growth}. From this figure, we can see that $k_i$ seems to remain bounded as the stochastic dimension $m$ increases, which agrees our analysis in section~\ref{interpretation}. 
\begin{figure}[htpb]
	\centering
	\begin{subfigure}{.5\textwidth}
		\centering
		\includegraphics[width=\textwidth]{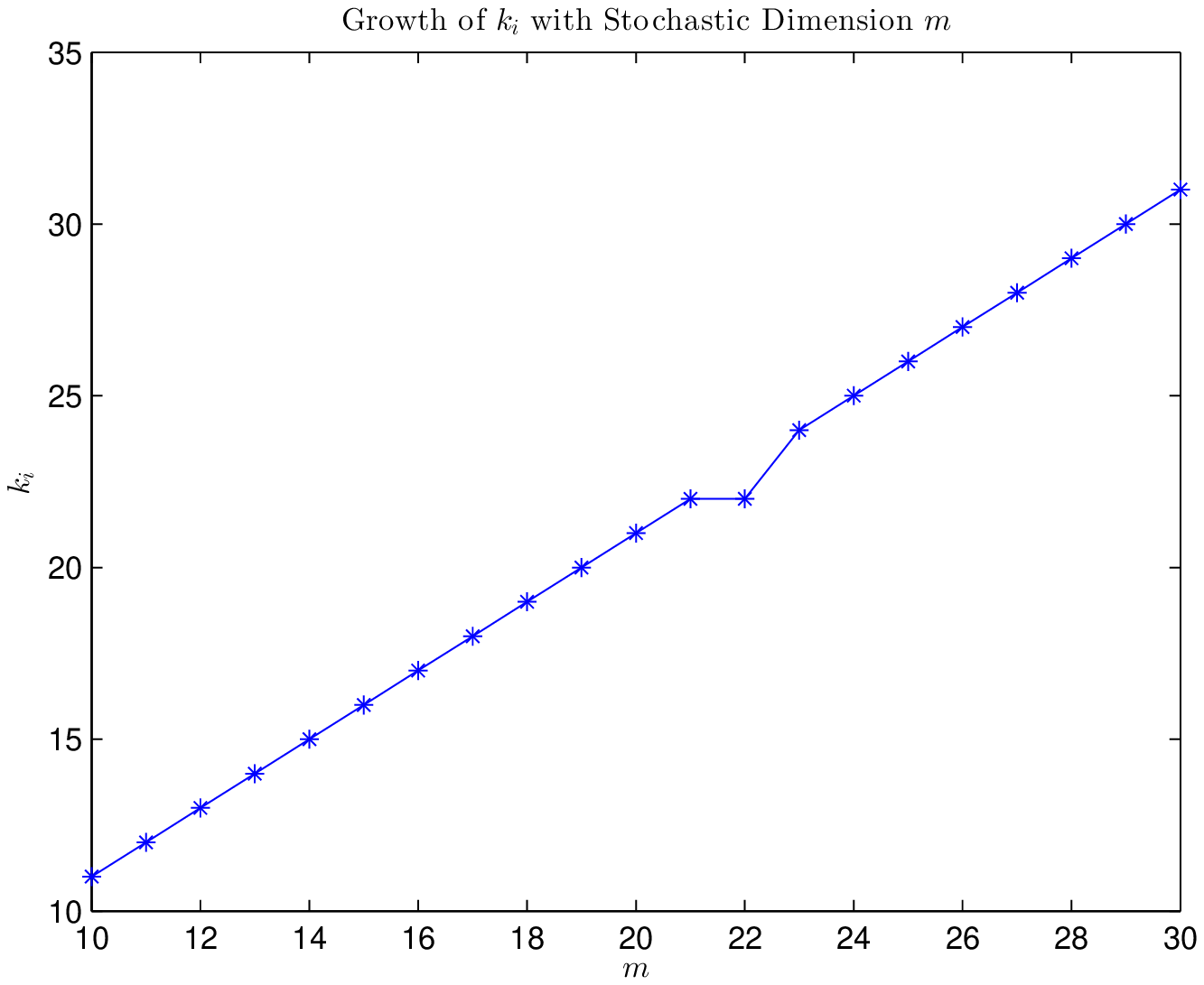}
		\caption{Unnormalized model.}
		\label{fig:k_growth}
	\end{subfigure}%
	\begin{subfigure}{.5\textwidth}
		\centering
		\includegraphics[width=\textwidth]{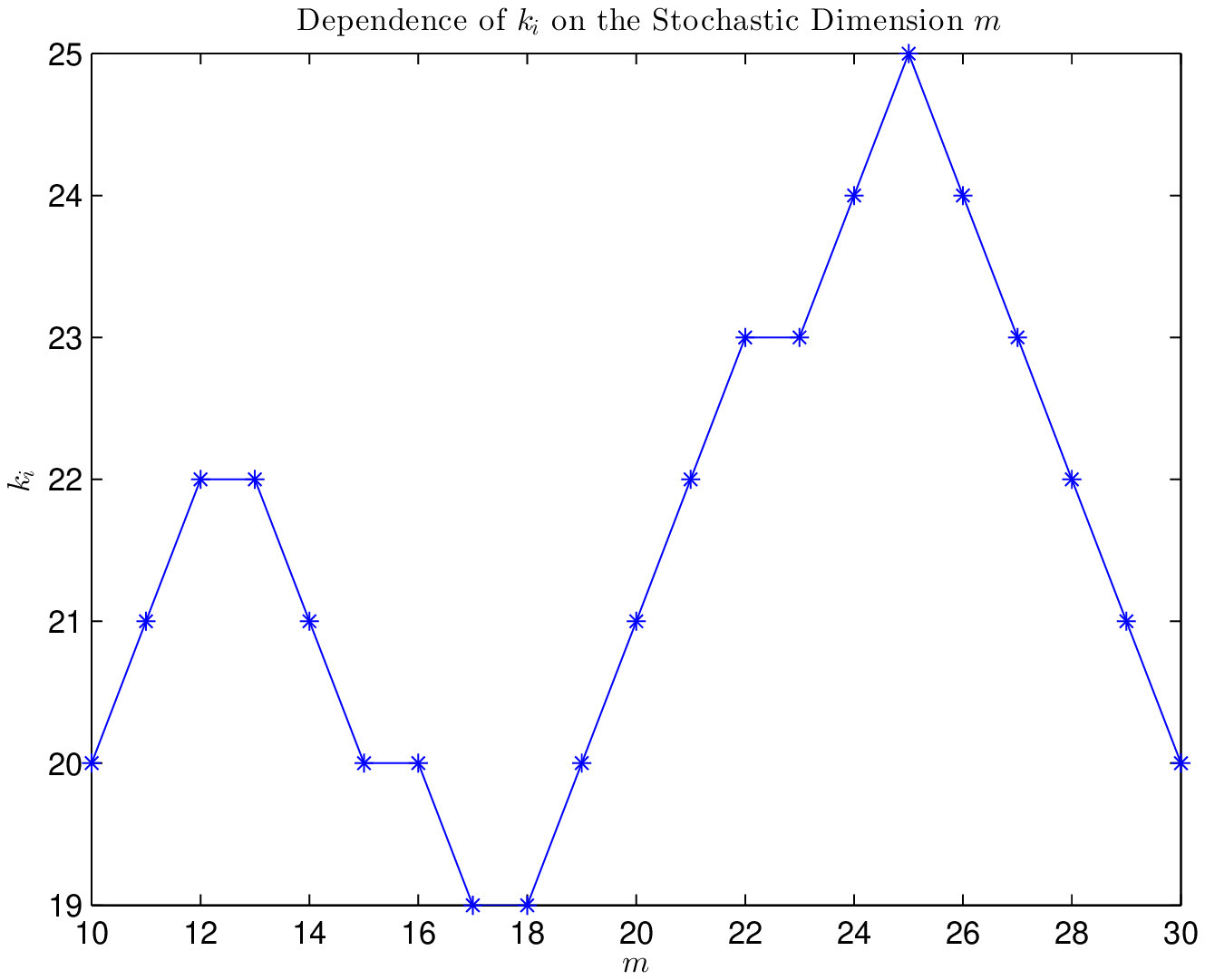}
		\caption{Normalized model.}
		\label{fig:k_no_growth}
	\end{subfigure}
	\caption{Dependence of $k_i$ on $m$.}
\end{figure}
\subsubsection{Performance of the Heterogeneous Stochastic Finite Element Method}
\label{performance}
In this section we apply the HSFEM to solve equation~\eqref{onedequation} with $m=20$ in \eqref{1dcoef}. We use the same discretization in the spatial and stochastic directions as the previous subsection. In the offline stage we discretize $L^2(D)$ as in~\eqref{discretel2} and the $\epsilon$, $r$ and $K$ in Algorithm~\ref{localbasis} are chosen to be
\[ \epsilon= 10\sqrt{2/\pi}\times10^{-3},\quad K=50,\quad r=5.\]

The average number of local stochastic basis functions constructed in the offline stage is
\begin{equation*}
	k=\frac{1}{n}\sum_{i=1}^nk_i\approx 26.2.
\end{equation*}
For this $m=20$ dimensional problem, $k=26.2$ is quite small compared with the {\em curse of dimensionality} suffered from by the PC and SC methods. The distribution of $k_i$ is plotted in Figure~\ref{fig:nlbs}.
\begin{figure}[htpb]
	\begin{center}
		\includegraphics[width=0.8\textwidth]{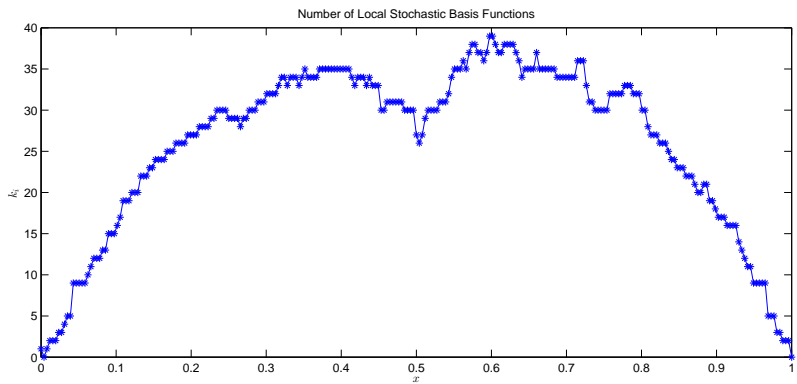}
	\end{center}
	\caption{Distribution of local stochastic basis constructed offline.}
	\label{fig:nlbs}
\end{figure}

From this figure we can see that more local stochastic basis functions are constructed in the interior of the domain than near the boundary. This is because we have chosen deterministic (actually homogeneous) boundary condition in equation~\eqref{onedequation}, and the solution has less randomness near the boundary.

In the online stage we solve equation \eqref{onedequation} with forcing function given by
\begin{equation} 
	f(x)=1-x+x^2-x^3.
	\label{force11}
\end{equation}

The expectation and standard deviation of the solution, and the numerical errors in these two quantities are plotted in Figure~\ref{fig:onedrf}. We can see that our method attains high accuracy in both mean and SD of the solution, which reflects that the local stochastic basis constructed using the sampling method in section~\ref{samplingsection} works well within the HSFEM framework.
\begin{figure}[htpb]
	\begin{center}
		\includegraphics[width=\textwidth]{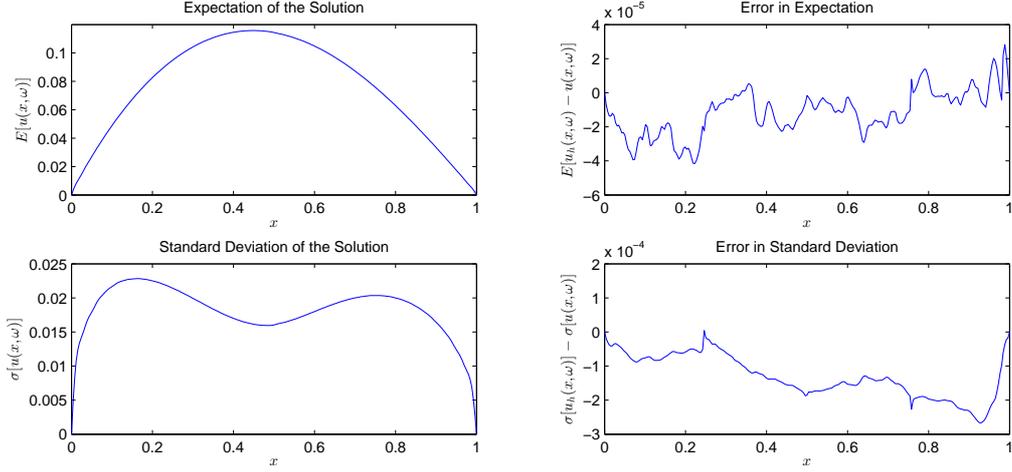}
	\end{center}
	\caption{Numerical error in Expectation and Standard Deviation of the solution. }
	\label{fig:onedrf}
\end{figure}

The relative $L^2(D\times\Omega)$ error $E_{HSFEM}$ ~\eqref{rl2} and relative KL truncation error $E_{KL}$ \eqref{rte} are listed in Table~\ref{tab:onedsf}. $E_{HSFEM}=1.83\times 10^{-2}$ is small, which means the HSFEM attains high accuracy in capturing the stochastic part of the solution. For this problem, the KL truncation error $E_{KL}$ is of the same order as $E_{HSFEM}$. As we have argued in section~\ref{subtlecompare}, this reflects the spatially heterogeneous stochastic structure of the solution and the advantage of approximating the stochastic part of the solution locally. 

\begin{table}
	\centering
	\begin{tabular}{|c|c|c|c|}
	\hline
	$k$ & Relative $L^2(D\times \Omega)$ Error $E_{HSFEM}$ & Relative KL truncation Error $E_{KL}$\\
	\hline
	$26.2$ & $1.83\times 10^{-2}$ &  $1.81\times 10^{-2}$\\
	\hline
	\end{tabular}
	\caption{Relative $L^2(D\times\Omega)$ error $E_{HSFEM}$ and relative KL truncation error $E_{KL}$.}
	\label{tab:onedsf}
\end{table}
\subsubsection{Convergence rate of the HSFEM }
The parameter $\epsilon$ in Algorithm~\ref{localbasis} affects the number of the returned local stochastic basis functions and the accuracy of the online numerical solution. To study how does $\epsilon$ affect the numerical error in the online stage, we consider solving the same equation as the previous subsection using different $\epsilon$ in the offline stage. Since the HSFEM is randomized in nature, the $k$ and $L^2(D\times\Omega)$ Error are both random. To reduce the fluctuation, for each $\epsilon$, we construct the local stochastic basis and solve the equation for $10$ times. The averaged $k$ and $L^2(D\times\Omega)$ Error of the $10$ times for different $\epsilon$ are plotted in Figure~\ref{fig:e_ek}.

\begin{figure}[htpb]
	\centering
		\begin{subfigure}{0.5\textwidth}
			\centering
			\includegraphics[width=\textwidth]{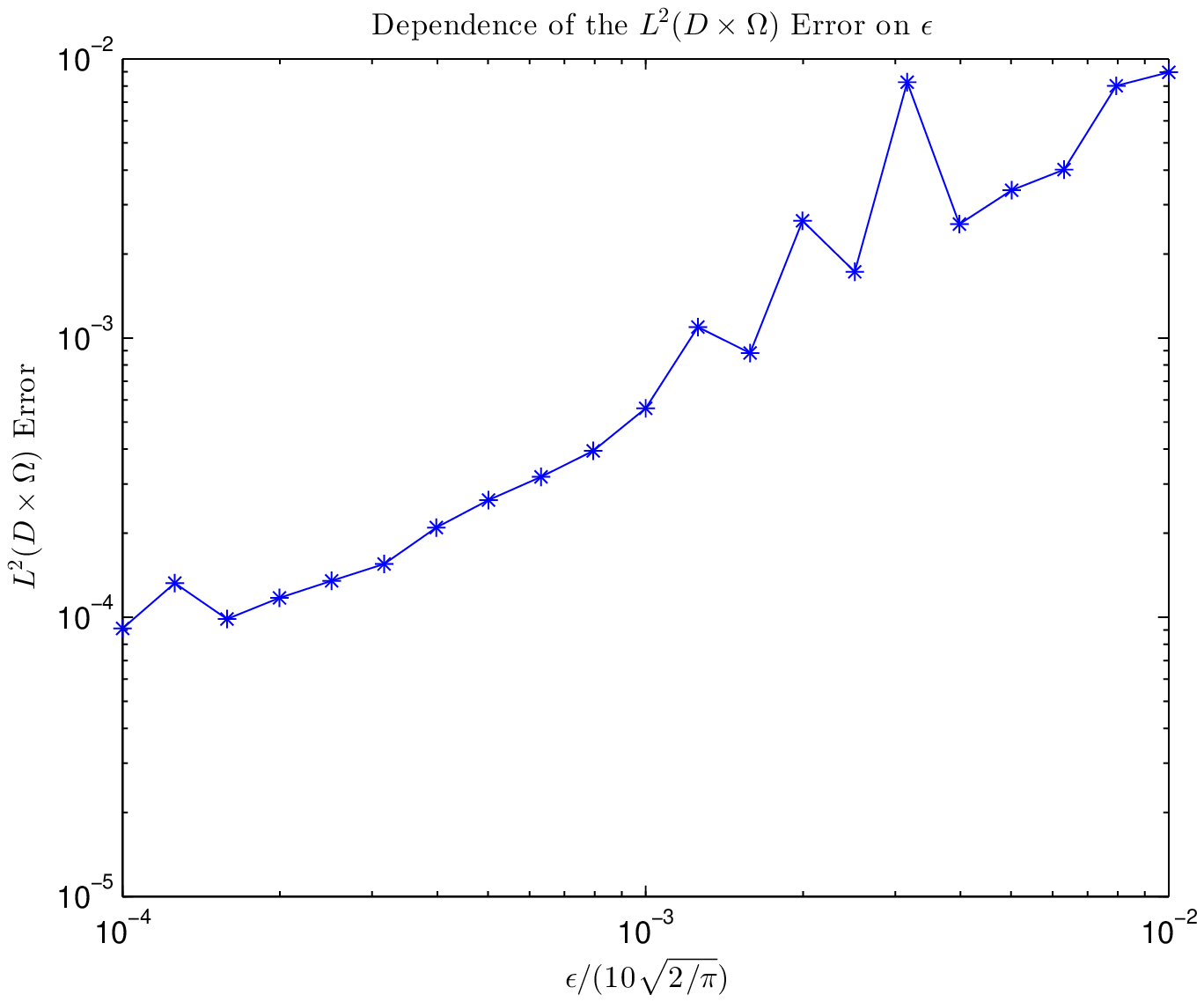}
			\caption{$L^2(D\times\Omega)$ Error with $\epsilon$.}
			\label{eestudy}
		\end{subfigure}%
		\begin{subfigure}{.5\textwidth}
			\centering
			\includegraphics[width=\textwidth]{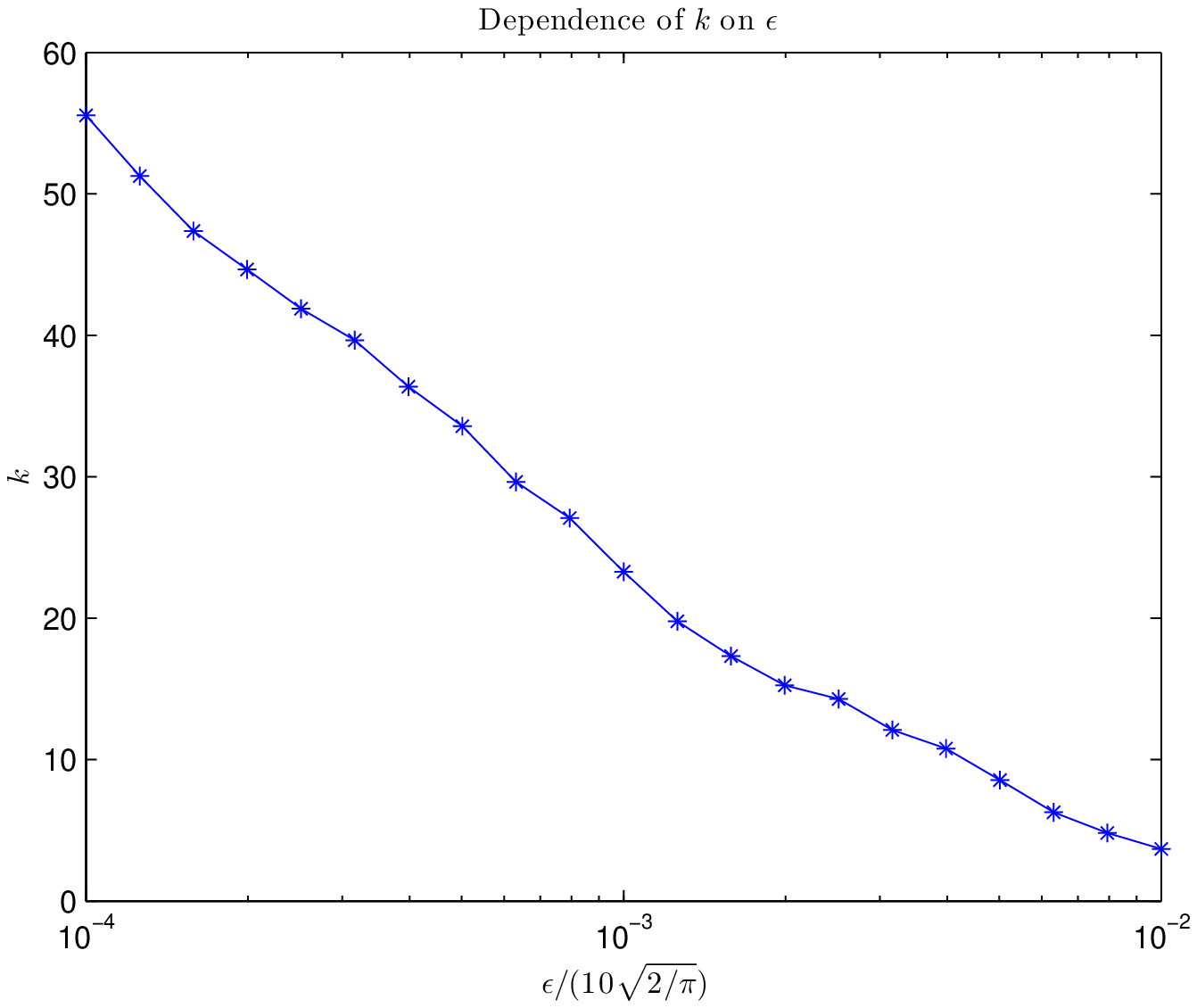}
			\caption{$k$ with $\epsilon$.}
			\label{ekstudy}
		\end{subfigure}%
		\caption{Dependence of $L^2(D\times\Omega)$ Error and $k$ on the parameter $\epsilon$.}
	\label{fig:e_ek}
\end{figure}

From Figure~\ref{eestudy}, we can see that the $L^2(D\times\Omega)$ error, which we denote by $E$, grows about linearly with $\epsilon$. By doing linear regression $\log E=\log C+\alpha\log \epsilon$, we obtain that
\begin{equation*}
	E\approx C\times \epsilon^{1.08},\quad C\approx 6.8\times 10^{-2}.
\end{equation*}
The linear growth of the $L^2(D\times\Omega)$ error with $\epsilon$ agrees with the approximation property \eqref{trialspaceproperty}. 

From Figure~\ref{ekstudy}, we can see that the returned $k$ grows about linearly with $-\log\epsilon$ as \eqref{mildgrowth}, which is much milder than~\eqref{worst} because the singular values of $T_i$ actually decay exponentially fast.

\subsubsection{The Online Error Estimation and correction Procedure}
\label{correctionsection}
In this subsection, we illustrate the implementation of the strategies to reduce the offline computational cost and the online error estimation and correction procedure introduced in section~\ref{correction}. We consider the stochastic elliptic operator given by \eqref{1dcoef} with $m=30$. For this problem with very high dimensional stochastic input, a very large number of collation points are required to obtain accuracy even if the Smolyak sparse grid quadratures are employed. So we instead use Monte Carlo samplings to discretize the probability space $(\Omega,P)$. We generate $M=4\times 10^4$ samples of $\omega$ according to its distribution, and store them as $\theta^1,\dots, \theta^M$. We set the weights $w^i=\frac{1}{M}$, then $(\Omega, P)$ is discretized to $(\theta^i,w^i)$. With this discretization in the stochastic direction, the local stochastic basis functions are stored as length-$M$ vectors, and the online numerical solutions are also defined on these sampling points. We can compute the statistical quantities of the numerical solutions based on their values on the sampling points and the weights $w^i$, but cannot recover the whole solution using polynomial interpolation as the previous cases where the stochastic collocation points are used to discretize $(\Omega,P)$. 

In the offline stage, we discretize the domain $D=[0,1]$ using a uniform coarse mesh of size $h=1/128$ to construct the local stochastic basis. In Algorithm~\ref{localbasis}, we choose
\[\epsilon/10\sqrt{2/\pi}=3\times10^{-3},\quad K=35,\quad r=5.\]
We first solve equation~\eqref{1dcoef} with randomly chosen forcing functions using the coarse mesh spatial basis, and then restrict the sampling solutions to the nodes of a uniform fine mesh of size $h=1/256$ to construct the local stochastic basis using Algorithm~\ref{localbasis}. 

The average number of local stochastic basis functions constructed in the offline stage is 
\[k=\frac{\sum_{i=1}^nk_i}{n}=24.8.\]
This $k$ is small, which means in the online stage the linear system that we need to solve is small. The distribution of $k_i$ is plotted in Figure~\ref{fig:fail}. We can see that for about half of the node points, we have $k_i=K$. This is because in Step 7-14 of Algorithm~\ref{localbasis}, all the left singular vectors of $A\Omega$ cannot make the condition in Step 10 hold. And to control the offline computation cost, we did not choose a larger $K$ as in Step 15-18, but simply return the algorithm with $Q=[U_1,\dots, U_K]$.
\begin{figure}[htpb]
	\begin{center}
		\includegraphics[width=0.8\textwidth]{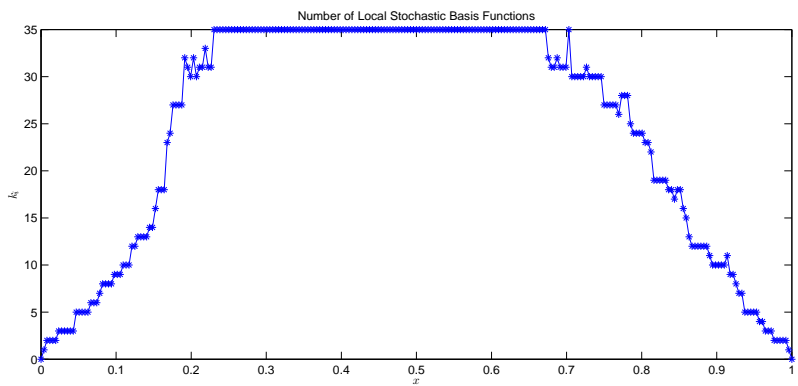}
	\end{center}
	\caption{Distribution of local stochastic basis.}
	\label{fig:fail}
\end{figure}

The mean and SD of the solution and the numerical errors in these two quantities are plotted in Figure~\ref{fig:1dmc}. We can see that there are relatively larger errors in these two quantities, and this is because the strategies introduced in section~\ref{correction} are taken to reduce the offline computational cost and the constructed local stochastic basis is not as accurate. For this case, the online error estimation and correction procedure is needed to obtain faithful numerical results.
\begin{figure}[htpb]
	\begin{center}
		\includegraphics[width=0.8\textwidth]{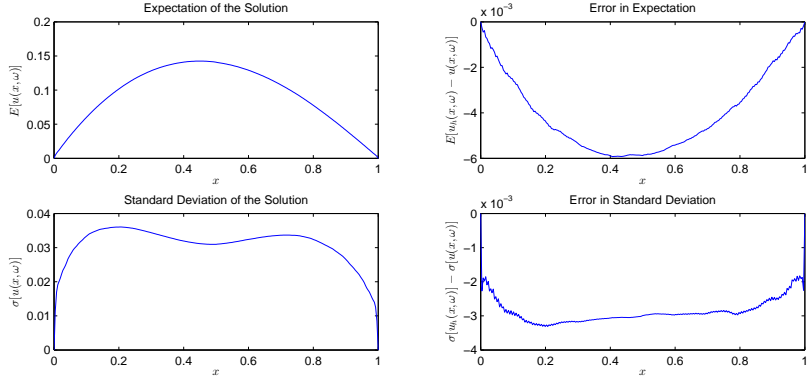}
	\end{center}
	\caption{Numerical error in Expectation and Standard Deviation of the solution.}
	\label{fig:1dmc}
\end{figure}

To illustrate the online error estimation and correction procedure, we consider the second order moment of the solution at $x=1/2$, i.e., we consider 
\[g(u)=u(1/2,\omega)^2.\]
Then based on the online numerical solution using the HSFEM, we have
\begin{equation}
	E[g(u_h)]=1.899\times 10^{-2}.
	\label{eng}
\end{equation}
And for the exact solution $u(x,\omega)$, which is the MC solution using the sampling points $\theta^i$ since we do not consider discretization error here, we have
\begin{equation}
	\label{eeg}
	E[g(u)]=2.078\times 10^{-2}.
\end{equation}

We can see that the error in $E[g(u_h)]$ is non-negligible. In practice, we will not compute the exact solution $u(x,\omega)$ because of the expensive computational cost, but instead apply the online error correction procedure given in section~\ref{correction} to correct $E[g(u_h)]$ and get a better approximation of $E[g(u)]$.

Since the HSFEM online numerical solution $u_h(x,\omega)$ is only defined on the discrete sampling points $\theta^i$, we cannot apply MC method to the original equation~\eqref{1dcoef} to correct error as in \eqref{rmc}. Instead, we apply the error estimation and correction procedure to the discretized problem with $\Omega$ replaced by $\{\theta^1, \dots, \theta^M\}$ and $P$ replaced by $w^1, \dots, w^M$. We choose $N_{MC}=100$ in \eqref{rmc}, and generate $N_{MC}$ independent random integers $i$ uniformly distributed from $1$ to $M$ (because $w^i$ are all equal). We denote them as $i_{n_{mc}}, n_{mc}=1,\dots, N_{MC}$ and then solve equation \eqref{1dcoef} on these sampling point $\theta^{i_{n_{mc}}}$. The average and variance of $\tau(\omega)$ \eqref{mcmcestimate} are approximately
\begin{equation*}
	E[g(u_h)]+\bar \tau=2.062\times 10^{-2},\quad \tilde \tau=1.025\times 10^{-4}.
\end{equation*}
Thus we get an approximate $95\%$ level confidence interval for $E[g(u)]$,
\begin{equation}
	\label{mcmc}
	[2.042\times 10^{-2}, 2.082\times 10^{-2}].
\end{equation}
If we instead apply MC methods directly to the discretized equation \eqref{1dcoef} using the same $N_{MC}=100$ in \eqref{mcestimator} , we get 
\begin{equation*}
	E[g(u)]\approx 2.033\times 10^{-2},\quad \sigma[g(u)]=9.210\times 10^{-4}.
\end{equation*}
and an approximate $95\%$ level confidence interval for $E[g(u)]$,
\begin{equation}
	\label{dmc}
	[1.805\times 10^{-2}, 2.261\times 10^{-2}].
\end{equation}
We can see that the approximation interval \eqref{mcmc} is much more accurate than \eqref{dmc}, and this is because the HSFEM solution $u_h(x,\omega)$ can capture the main stochastic part of the solution, and $g(u)-g(u_h)$ has much smaller variance than $g(u)$. Using the discretized exact solution, we can compute that
\[\sigma[g(u)]=1.2361\times 10^{-2},\quad \sigma[g(u)-g(u_h)]=2.3131\times 10^{-3}.\]

This example demonstrates the effectiveness of the HSFEM as a tool for variance reduction. We emphasize that, in practice, this online error estimation and correction procedure can be used to get more faithful numerical results even if the strategies to reduce offline cost are not taken. 
\subsection{A 2D Example With Gaussian Random Variables}
In this section, we consider the following 2D SPDE
\begin{equation} 
	\begin{cases}
	-{\rm div}(a(x,y,\omega)\nabla u(x,y,\omega))=f(x,y),\quad (x,y)\in D=(0,1)^2, \quad \omega \in \Omega,\\
	u(x,\omega)|_{\partial D}=0.
	\end{cases}
	\label{2dequation}
\end{equation} 

The coefficient $a(x,\omega)$ is given by 
\begin{equation*}
	\log( a(x,y,\omega))=1+\frac{1}{4}\sum_{k=1}^{12}\omega_k(\sin(k\pi x)+\cos((13-k)\pi y)),
	\label{2dcoef}
\end{equation*}
where $\omega_i, i=1,\dots, 12$ are independent Gaussian random variables, $\omega_i\sim\mathcal{N}(0,1)$. For this coefficient $a(x,\omega)$, the uniform ellipticity condition \eqref{ellipticity} is actually violated. But in our implementation, we discretize $\Omega$ using (finite) sampling points, and the discretized problem is still elliptic and well-posed.

None of the $12$ random variables is negligible in the coefficient thus this stochastic operator has genuine high stochastic dimension. In the spatial direction, the domain $D$ is discretized using a standard right triangular mesh of size $h=1/64$, resulting in $2\times 64^2$ elements. In the stochastic direction, we discretize the problem using the 4th order Smolyak sparse grid collocation points, getting a total of $M=3361$ sampling points $\theta^i$. In the offline stage we choose $l=32$ in \eqref{discretizeL2D}, then $\hat D$ contains all the Fourier modes that can resolved by the given mesh. In Algorithm~\ref{localbasis}, we choose  
\[ K=50,\quad \epsilon/(10\sqrt{2/\pi})=3\times 10^{-4}, \quad r=5.\] 
The average number of local stochastic basis functions constructed in the offline stage is
\begin{equation*}
	k=\frac{1}{n}\sum_{i=1}^nk_i\approx16.5,
\end{equation*}
which is small given that the stochastic input has high dimension $m=12$. The maximum number of local stochastic basis functions, $\max_i k_i=41$, is significantly larger than $k$. This reveals that the solution space has significantly richer stochastic structure in some regions of the domain than others, and this heterogeneous stochastic structure of the solution space is recognized by the HSFEM.

The forcing function we choose in the online stage is 
\begin{equation*}
	f(x,y)=1+x-2y.
\end{equation*}

The expectation and standard deviation of the solution, and the numerical errors in these two quantities are plotted in Figure~\ref{fig:2dexpsforceresult}. From this figure we can see that the HSFEM attains high accuracy in both expectation and standard deviation of the solution.
\begin{figure}[htpb]
	\begin{center}
		\includegraphics[width=0.8\textwidth]{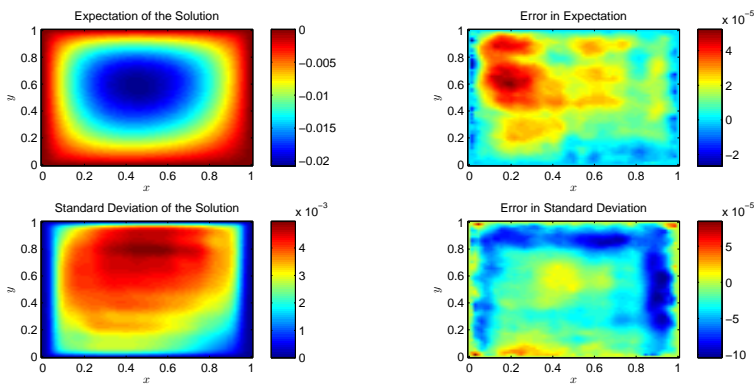}
	\end{center}
	\caption{Error in Expectation and Standard Deviation of the solution. }
	\label{fig:2dexpsforceresult}
\end{figure}
The relative $L^2(D\times\Omega)$ error in the stochastic part of the solution $E_{HSFEM}$, and relative KL truncation error $E_{KL}$ are listed in Table~\ref{tab:2dexpsf}. $E_{HSFEM}=4.5\times 10^{-2}$ is small, which means our method attains good accuracy in capturing the stochastic part of the solution, $u(x,\omega)-\bar u(x)$. For this problem, $E_{HSFEM}$ is significantly less than $E_{KL}$, and as we have argued, this reveals the heterogeneous stochastic structure of the solution space and the advantage of approximating the stochastic behavior of the solution using local stochastic basis.

\begin{table}
	\centering
	\begin{tabular}{|c|c|c|c|}
	\hline
	$k$ & $E_{HSFEM}$ & $E_{KL}$\\
	\hline
	$16.5$ & $4.46\times 10^{-2}$ &  $9.60\times 10^{-2}$\\
	\hline
	\end{tabular}
	\caption{The relative $L^2(D\times\Omega)$ error and the relative KL truncation error $E_{KL}$. }
	\label{tab:2dexpsf}
\end{table}

\subsection{A 2D Example with Discontinuous Coefficients}
In this section, we consider solving equation~\eqref{2dequation} with $a(x,y,\omega)$ discontinuous. We divide $D=[0,1]\times[0,1]$ to 4 regions, which are
\[D_1=[0,1/2]\times[0,1/2],\quad D_2=[1/2,1]\times[0,1/2],\quad D_3=[0,1/2]\times[1/2,1],\quad D_4=[1/2,1]\times[1/2,1].\]
And the coefficient $a(x,y,\omega)$ is given by
\begin{equation}
	\label{jumpcoef}
	\log(a(x,y,\omega))=
	\begin{cases}
		\sum_{k=1}^3\omega_k(\sin(2k\pi x)+\cos(2(4-k)\pi y)), \quad & (x,y)\in D_1,\\
		1+\sum_{k=4}^6\omega_k(\sin(2(k-3)\pi x)+\cos(2(7-k)\pi y)),\quad & (x,y)\in D_2,\\
		2+\sum_{k=7}^9\omega_k(\sin(2(k-6)\pi x)+\cos(2(10-k)\pi y)),\quad & (x,y)\in D_3,\\
		3+\sum_{k=10}^{12}\omega_k(\sin(2(k-9)\pi x)+\cos(2(13-k)\pi y)),\quad & (x,y)\in D_4,
	\end{cases}
\end{equation}
where $\omega_i, i=1,\dots, 12$ are independent standard Gaussian random variables, $\omega_i\sim\mathcal{N}(0,1)$.

None of the random variables $\omega_i$ in \eqref{jumpcoef} is negligible, so this problem has real high stochastic dimension. The discretization of the problem in the spatial and stochastic directions are the same as the previous example. And in the randomized range finding Algorithm~\ref{localbasis} we choose
\[K=50, \quad \epsilon/(10\sqrt{2/\pi})=10^{-4},\quad r=5.\]

The average number of local stochastic basis functions constructed in the offline stage is
\begin{equation*}
	k=\frac{1}{n}\sum_{i=1}^nk_i\approx 12.3,
\end{equation*}
which can be considered as small given that the stochastic input has high dimension $m=12$.

The maximum number of local stochastic basis functions $\max_i k_i=40$ is significantly larger than $k$, which reflects the heterogeneous stochastic structure of the solution space. The distribution of $k_i$ is plotted in Figure~\ref{fig:2dlb}. We can see that more local stochastic basis functions are put in $D_1$, which according to \eqref{jumpcoef}, is the region where the random part of the coefficient $a(x,\omega)$ has stronger effect.  

\begin{figure}[htpb]
	\begin{center}
		\includegraphics[width=0.5\textwidth]{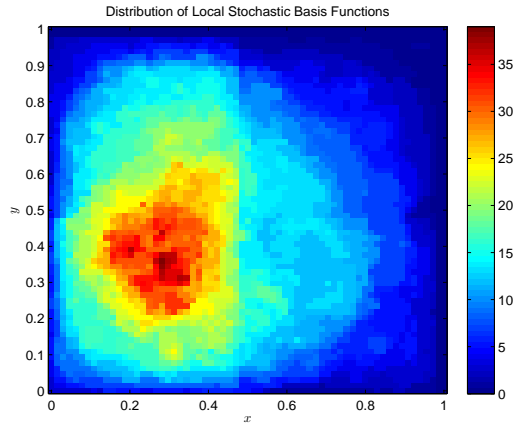}
	\end{center}
	\caption{Distribution of local stochastic basis for the example with discontinuous coefficient.}
	\label{fig:2dlb}
\end{figure}

The forcing function we choose in the online stage is $f(x,y)=1+x-2y$. The expectation and standard deviation of the solution, and the numerical errors in these two quantities are plotted in Figure~\ref{fig:disreult}. 
\begin{figure}[htpb]
	\begin{center}
		\includegraphics[width=0.8\textwidth]{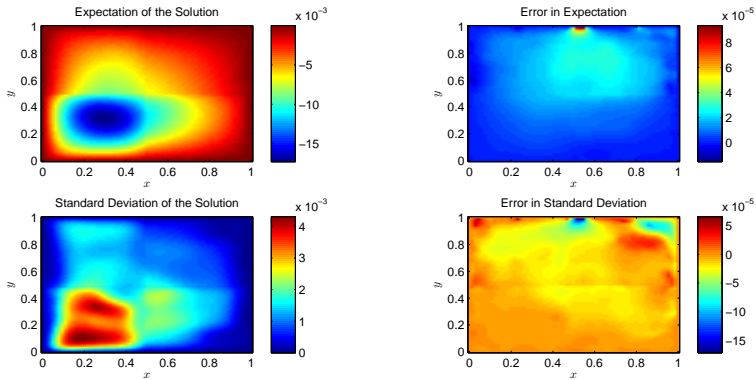}
	\end{center}
	\caption{Error in Expectation and Standard Deviation of the solution.} 
	\label{fig:disreult}
\end{figure}
The relative error in stochastic part of the solution $E_{HSFEM}$, and relative KL truncation error $E_{KL}$ are listed in Table~\ref{tab:2djump}. We can see that our method attains good accuracy for this problem with discontinuous coefficients. Again, we find that $E_{HSFEM}$ is significantly smaller than $E_{KL}$. 
\begin{table}
	\centering
	\begin{tabular}{|c|c|c|c|}
	\hline
	$k$ & $E_{HSFEM}$ & $E_{KL}$\\
	\hline
	$15.2$ & $3.3\times 10^{-2}$ &  $1.5\times 10^{-1}$\\
	\hline
	\end{tabular}
	\caption{The relative $L^2(D\times\Omega)$ error $E_{HSFEM}$ and the relative KL truncation error $E_{KL}$.}
	\label{tab:2djump}
\end{table}
\subsection{A 2D Problem With Very High Dimensional Stochastic Input}
\label{study}
In this subsection we study a 2D problem with very high dimensional Gaussian input. We denote the following set of orthonormal $L^2([0,1]^2)$ functions
\[\{2\sin(2\pi x),\dots, 2\sin(2k\pi x),\dots, 2\sin(12\pi x)\}\otimes\{2\cos(2\pi y),\dots, 2\cos(2k\pi y),\dots, 2\cos(12\pi y)\},\]
as $\{\Psi_1(x,y),\Psi_2(x,y)\dots, \Psi_{36}(x,y)\}$, and consider operator $-{\rm div}\left(a(x,y,\omega)\nabla(\cdot)\right)$ with $a(x,y,\omega)$ given by 
\[\log (a(x,y,\omega))=\frac{1}{2}\sum_{i=1}^{36}\omega_i\Psi_i(x,y),\]
where $\omega_i, i=1,\dots, 36$ are independent standard Gaussian random variables, $\omega_i\sim\mathcal{N}(0,1)$. This operator has genuine high dimensional stochastic input, and $\omega_i, i=1,\dots, 36$, contribute equally to $\log(a_k(x,y,\omega))$ in $L^2(D\times\Omega)$. We discretize the domain using a standard right triangular mesh of size $h=\frac{1}{32}$, and $D$ is divided into $2\times 32^2$ elements. We choose $l=16$ in the discretization of $L^2(D)$~\eqref{discretizeL2D}, then $\hat D$ contains all the Fourier modes that can be resolved by this mesh. In the stochastic direction, since the dimension of $\omega$ is $36$, which is very high, a very large number of collocation points are required even if we use the Smolyak sparse grid quadrature. So we instead discretize $\Omega$ using $M=10^4$ Monte Carlo sampling points and set the weights $w^i=1/M$. In Algorithm~\ref{localbasis}, we choose $K=70$, $\epsilon/(10\sqrt{2/\pi})=2\times 10^{-3}$, $r=5$. The average number of the local stochastic basis functions constructed offline is
\begin{equation*}
	k=17.6.
\end{equation*}
This $k$ is quite small given that this problem has very high stochastic dimension $m=36$. In addition, $\max_ik_i=51$ is significantly larger than $k$, which implies the solution space has strong heterogeneous stochastic structure, and this is recognized by the HSFEM.

In the online stage, we solve the equation using forcing $f(x)=1+x-2y$. The expectation and standard deviation of the solution and numerical errors in these two quantifies are plotted in Figure~\ref{fig:mc_s_forcing}. 
\begin{figure}[htpb]
	\begin{center}
		\includegraphics[width=0.8\textwidth]{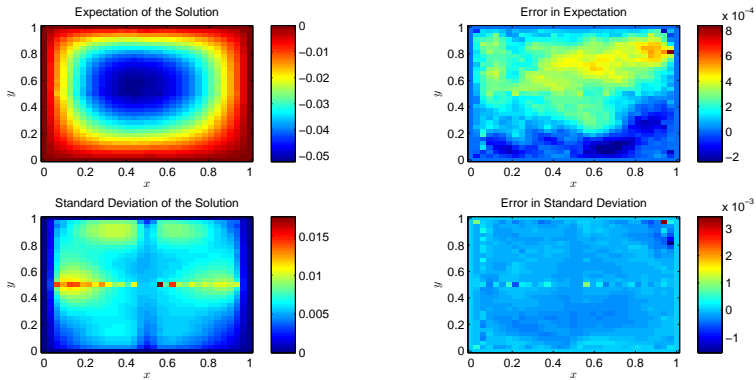}
	\end{center}
	\caption{Expectation and Standard Deviation of the solution.}
	\label{fig:mc_s_forcing}
\end{figure}
The relative numerical error $E_{HSFEM}$ and relative KL truncation error are listed in Table~\ref{tab:mc_result}. $E_{HSFEM}=1.04\times 10^{-1}$, which means the HSFEM numerical solution captures about $90\%$ of the stochastic part of the solution. The numerical error $E_{HSFEM}$ is significantly smaller than the KL truncation error $E_{KL}$. 
 
\begin{table}
	\centering
	\begin{tabular}{|c|c|c|c|}
	\hline
	$k$ & $E_{HSFEM}$ & $E_{KL}$\\
	\hline
	$17.6$ & $1.04\times 10^{-1}$ &  $4.42\times 10^{-1}$\\
	\hline
	\end{tabular}
	\caption{The relative $L^2(D\times\Omega)$ error $E_{HSFEM}$ and the relative KL truncation error $E_{KL}$.}
	\label{tab:mc_result}
\end{table}

Recall that in computing the numerical errors, we did not consider the discretization error in~\eqref{decomposeerror}, and we were comparing our numerical solutions with the MC or SC solutions. In this example, the number of MC sampling points $M=10^4$ is actually not very large, and the discretization error in~\eqref{decomposeerror} may be dominant. Nevertheless, our numerical results still demonstrate the capacity of the HSFEM in exploiting the local stochastic structure of the solution space to the discretized problem.

\section{Concluding Remarks}
\label{concludingsection}
A new concept of sparsity has been introduced for linear stochastic elliptic operator, which reflects the compactness of its inverse operator in the stochastic direction and allows for spatially heterogeneous stochastic structure of the corresponding solution space. This new concept of sparsity motivates a HSFEM framework for solving linear stochastic elliptic equations, which uses a problem-dependent and local stochastic basis to approximate the solution. This HSFEM framework provides a novel direction to attack the {\em curse of dimensionality} by exploiting the local stochastic structure of the solutions. 

Constructing a suitable local stochastic basis is a challenging task since the inverse of the elliptic operator is highly implicit, nonlinear and non-local. In this work we provide a sampling method to construct the local stochastic basis using the randomized range finding methods. It involves sampling the stochastic operator for $K$ times using randomly chosen forcing functions and some orthogonalization process. Hopefully, if the stochastic operator enjoys the Operator-Sparsity, a small $K$ is sufficient to identify the local stochastic structure of the solution space. The relatively expensive offline computation limits the application of the HSFEM to the multi-query setting. Methods to reduce the offline computational cost and an online error estimation and correction procedure are given.

Numerical results are presented to demonstrate the spatially heterogeneous stochastic structure of the solution space to stochastic elliptic equations, and the Operator-Sparsity for several elliptic operators with high dimensional stochastic input. The proposed HSFEM can recognize and respect the heterogeneous stochastic structure of the solution space, thus achieve high efficiency in the online stage.

We only consider Dirichlet problem in this paper, but our method can be easily generalized to deal with Neumann boundary condition problem. The key idea of the present work is the use of the heterogeneous coupling of spatial basis with local stochastic basis to approximate the solution, which should not be restricted to elliptic equations, and will be applied to time-dependent problems in our future work. Convergence analysis of the HSFEM framework will be given in another paper.

\section{Acknowledgements}
We would like to thank the two anonymous reviewers for their constructive comments which help improve the quality of this paper. This research was in part supported by Air Force MURI Grant FA9550-09-1-0613, DOE grant DE-FG02-06ER257, and NSF Grants No. DMS-1318377, DMS-1159138. 
\bibliographystyle{plain}
\bibliography{StochasticMethod}

\begin{thebibliography}{10}

\bibitem{babuvska2002perturb}
Ivo Babu{\v{s}}ka and Panagiotis Chatzipantelidis.
\newblock On solving elliptic stochastic partial differential equations.
\newblock {\em Computer Methods in Applied Mechanics and Engineering},
  191(37):4093--4122, 2002.

\bibitem{babuvska2010stochastic}
Ivo Babu{\v{s}}ka, Fabio Nobile, and Raul Tempone.
\newblock A stochastic collocation method for elliptic partial differential
  equations with random input data.
\newblock {\em SIAM review}, 52(2):317--355, 2010.

\bibitem{babuska2004galerkin}
Ivo Babuska, Ra{\'u}l Tempone, and Georgios~E Zouraris.
\newblock Galerkin finite element approximations of stochastic elliptic partial
  differential equations.
\newblock {\em SIAM Journal on Numerical Analysis}, 42(2):800--825, 2004.

\bibitem{bieri2011sparse}
Marcel Bieri.
\newblock A sparse composite collocation finite element method for elliptic
  spdes.
\newblock {\em SIAM Journal on Numerical Analysis}, 49(6):2277--2301, 2011.

\bibitem{bieri2009sparse1}
Marcel Bieri, Roman Andreev, and Christoph Schwab.
\newblock Sparse tensor discretization of elliptic spdes.
\newblock {\em SIAM Journal on Scientific Computing}, 31(6):4281--4304, 2009.

\bibitem{bieri2009sparse}
Marcel Bieri and Christoph Schwab.
\newblock Sparse high order fem for elliptic spdes.
\newblock {\em Computer Methods in Applied Mechanics and Engineering},
  198(13):1149--1170, 2009.

\bibitem{caflisch1998monte}
Russel~E Caflisch.
\newblock Monte carlo and quasi-monte carlo methods.
\newblock {\em Acta numerica}, 7:1--49, 1998.

\bibitem{cameron1947orthogonal}
Robert~H Cameron and William~T Martin.
\newblock The orthogonal development of non-linear functionals in series of
  fourier-hermite functionals.
\newblock {\em Annals of Mathematics}, pages 385--392, 1947.

\bibitem{cheng2013data}
Mulin Cheng, Thomas~Y Hou, Mike Yan, and Zhiwen Zhang.
\newblock A data-driven stochastic method for elliptic pdes with random
  coefficients.
\newblock {\em Journal on Uncertainty Quantification}, 1, 2013.

\bibitem{cheng2013dynamically1}
Mulin Cheng, Thomas~Y Hou, and Zhiwen Zhang.
\newblock A dynamically bi-orthogonal method for time-dependent stochastic
  partial differential equations i: Derivation and algorithms.
\newblock {\em Journal of Computational Physics}, 242:843--868, 2013.

\bibitem{cheng2013dynamically2}
Mulin Cheng, Thomas~Y Hou, and Zhiwen Zhang.
\newblock A dynamically bi-orthogonal method for time-dependent stochastic
  partial differential equations ii: Adaptivity and generalizations.
\newblock {\em Journal of Computational Physics}, 242:753--776, 2013.

\bibitem{doostan2007stochastic}
Alireza Doostan, Roger~G Ghanem, and John Red-Horse.
\newblock Stochastic model reduction for chaos representations.
\newblock {\em Computer Methods in Applied Mechanics and Engineering},
  196(37):3951--3966, 2007.

\bibitem{doostan2011non}
Alireza Doostan and Houman Owhadi.
\newblock A non-adapted sparse approximation of pdes with stochastic inputs.
\newblock {\em Journal of Computational Physics}, 230(8):3015--3034, 2011.

\bibitem{frauenfelder2005finite}
Philipp Frauenfelder, Christoph Schwab, and Radu~Alexandru Todor.
\newblock Finite elements for elliptic problems with stochastic coefficients.
\newblock {\em Computer methods in applied mechanics and engineering},
  194(2):205--228, 2005.

\bibitem{ghanem1991stochastic}
Roger~G Ghanem and Pol~D Spanos.
\newblock {\em Stochastic finite elements: a spectral approach}, volume~41.
\newblock Springer, 1991.

\bibitem{gruter1982green}
Michael Gr{\"u}ter and Kjell-Ove Widman.
\newblock The green function for uniformly elliptic equations.
\newblock {\em Manuscripta Mathematica}, 37(3):303--342, 1982.

\bibitem{hackbusch1999sparse}
Wolfgang Hackbusch.
\newblock A sparse matrix arithmetic based on h-matrices. part i: Introduction
  to h-matrices.
\newblock {\em Computing}, 62(2):89--108, 1999.

\bibitem{hackbusch2000sparse}
Wolfgang Hackbusch and Boris~N Khoromskij.
\newblock A sparse h-matrix arithmetic. part ii: application to
  multi-dimensional problems.
\newblock {\em Computing}, 64(1):21--47, 2000.

\bibitem{halko2011finding}
Nathan Halko, Per-Gunnar Martinsson, and Joel~A Tropp.
\newblock Finding structure with randomness: Probabilistic algorithms for
  constructing approximate matrix decompositions.
\newblock {\em SIAM review}, 53(2):217--288, 2011.

\bibitem{hou2006wiener}
Thomas~Y Hou, Wuan Luo, Boris Rozovskii, and Hao-Min Zhou.
\newblock Wiener chaos expansions and numerical solutions of randomly forced
  equations of fluid mechanics.
\newblock {\em Journal of Computational Physics}, 216(2):687--706, 2006.

\bibitem{ito1998reduced}
Kazufumi Ito and SS~Ravindran.
\newblock A reduced-order method for simulation and control of fluid flows.
\newblock {\em Journal of computational physics}, 143(2):403--425, 1998.

\bibitem{kleiber1992stochastic}
Michael Kleiber and Tran~Duong Hien.
\newblock {\em The stochastic finite element method: basic perturbation
  technique and computer implementation}.
\newblock Wiley New York, 1992.

\bibitem{math2012a}
Mathelin L. and Gallivan K.A.
\newblock A compressed sensing approach for partial differential equations with
  random input data.
\newblock {\em Communications in computational physics}, 12(4):919--954, 2012.

\bibitem{liberty2007randomized}
Edo Liberty, Franco Woolfe, Per-Gunnar Martinsson, Vladimir Rokhlin, and Mark
  Tygert.
\newblock Randomized algorithms for the low-rank approximation of matrices.
\newblock {\em Proceedings of the National Academy of Sciences},
  104(51):20167--20172, 2007.

\bibitem{lin2011fast}
Lin Lin, Jianfeng Lu, and Lexing Ying.
\newblock Fast construction of hierarchical matrix representation from
  matrix--vector multiplication.
\newblock {\em Journal of Computational Physics}, 230(10):4071--4087, 2011.

\bibitem{mathelin2005stochastic}
Lionel Mathelin, M~Yousuff Hussaini, and Thomas~A Zang.
\newblock Stochastic approaches to uncertainty quantification in cfd
  simulations.
\newblock {\em Numerical Algorithms}, 38(1-3):209--236, 2005.

\bibitem{niederreiter1992random}
Harald Niederreiter.
\newblock {\em Random number generation and quasi-Monte Carlo methods},
  volume~63.
\newblock SIAM, 1992.

\bibitem{nobile2008sparse}
Fabio Nobile, Ra{\'u}l Tempone, and Clayton~G Webster.
\newblock A sparse grid stochastic collocation method for partial differential
  equations with random input data.
\newblock {\em SIAM Journal on Numerical Analysis}, 46(5):2309--2345, 2008.

\bibitem{rozza2008reduced}
Gianluigi Rozza, DBP Huynh, and Anthony~T Patera.
\newblock Reduced basis approximation and a posteriori error estimation for
  affinely parametrized elliptic coercive partial differential equations.
\newblock {\em Archives of Computational Methods in Engineering},
  15(3):229--275, 2008.

\bibitem{sargsyan2014dimensionality}
Khachik Sargsyan, Cosmin Safta, Habib~N Najm, Bert~J Debusschere, Daniel
  Ricciuto, and Peter Thornton.
\newblock Dimensionality reduction for complex models via bayesian compressive
  sensing.
\newblock {\em International Journal for Uncertainty Quantification}, 4(1),
  2014.

\bibitem{smolyak1963quadrature}
Sergey~A Smolyak.
\newblock Quadrature and interpolation formulas for tensor products of certain
  classes of functions.
\newblock In {\em Dokl. Akad. Nauk SSSR}, volume~4, page 123, 1963.

\bibitem{stark1986probability}
Henry Stark and John~William Woods.
\newblock {\em Probability, random processes, and estimation theory for
  engineers}.
\newblock Prentice-Hall Englewood Cliffs (NJ), 1986.

\bibitem{steinbach2008}
Olaf Steinbach.
\newblock {\em Numerical approximation methods for elliptic boundary value
  problems: Finite and boundary elements}, volume~99.
\newblock Springer, 2008.

\bibitem{taylor2013green}
JL~Taylor, S~Kim, and RM~Brown.
\newblock The green function for elliptic systems in two dimensions.
\newblock {\em Communications in Partial Differential Equations},
  38(9):1574--1600, 2013.

\bibitem{wiener1938homogeneous}
Norbert Wiener.
\newblock The homogeneous chaos.
\newblock {\em Amer. J. Math}, 60(4):897--936, 1938.

\bibitem{xiu2005high}
Dongbin Xiu and Jan~S Hesthaven.
\newblock High-order collocation methods for differential equations with random
  inputs.
\newblock {\em SIAM Journal on Scientific Computing}, 27(3):1118--1139, 2005.

\bibitem{xiu2002wiener}
Dongbin Xiu and George~Em Karniadakis.
\newblock The wiener--askey polynomial chaos for stochastic differential
  equations.
\newblock {\em SIAM Journal on Scientific Computing}, 24(2):619--644, 2002.

\end{thebibliography}
\end{document}